\documentclass[12pt]{amsart}
\usepackage[all,cmtip]{xy}
\usepackage{amsmath,amscd,amsthm,amsfonts, amssymb,amsxtra}
\usepackage{graphicx}
\usepackage{hyperref}
  \hypersetup{colorlinks=true,citecolor=blue}
  \usepackage[font=scriptsize]{caption} 
\usepackage{tikz}
\usetikzlibrary{cd,decorations.markings}
\usepackage{subcaption}
\CDat
\SelectTips{cm}{}
\newtheorem{thm}{Theorem}[section]
\newtheorem*{un-no-thm}{Theorem}
\newtheorem{cor}[thm]{Corollary}     
\newtheorem{lem}[thm]{Lemma}         
\newtheorem{prop}[thm]{Proposition}

\newtheorem{hypo_no}[thm]{Hypothesis}

\newtheorem{bigthm}{Theorem}

\theoremstyle{definition}
\newtheorem{defn}[thm]{Definition}   

\theoremstyle{definition}

\theoremstyle{definition}

\theoremstyle{definition}

\theoremstyle{remark}
\newtheorem{rem}[thm]{Remark}

\newtheorem*{acks}{Acknowledgements}
\newtheorem*{out}{Outline}
\newtheorem{rems}[thm]{Remarks}

\newtheorem{ex}[thm]{Example}

\newcommand{\mM}{\mathcal M}

\begin{document}
\title[Fluctuations of cycles]{Fluctuations of cycles\\ in a finite CW complex}

\date{\today}

\author[M.~J.~Catanzaro]{Michael J.\ Catanzaro}
\address{Department of Mathematics, University of Florida, Gainesville, FL 32611}
\email{catanzaro@math.ufl.edu}
\author[V.~Y.~Chernyak] {Vladimir Y.\ Chernyak}
\address{Department of Chemistry, Wayne State University, Detroit, MI 48202}
\email{chernyak@chem.wayne.edu}
\author[J.~R.~Klein]{John R.\ Klein}
\address{Department of Mathematics, Wayne State University, Detroit, MI 48202}
\email{klein@math.wayne.edu}

\begin{abstract}
 We use algebraic topology to study
the stochastic motion of cellular cycles in a finite CW
complex. Inspired by statistical mechanics, we introduce a homological
 observable called the {\it average current}. The latter measures the average flux of  the probability in the process. 
In the low temperature, adiabatic limit, we prove that the  average current fractionally 
quantizes, in which the denominators are 
combinatorial invariants of the CW complex.
\end{abstract}

\maketitle
\setlength{\parindent}{15pt}
\setlength{\parskip}{1pt plus 0pt minus 1pt}
\def\bdot{\bold .}
\def\Sp{\bold S\bold p}

\def\vo{\varOmega}
\def\smsh{\wedge}
\def\^{\wedge}
\def\flush{\flushpar}
\def\id{\text{\rm id}}
\def\dbslash{/\!\! /}
\def\codim{\text{\rm codim\,}}
\def\:{\colon}
\def\holim{\text{holim\,}}
\def\hocolim{\text{hocolim\,}}
\def\cal{\mathcal}
\def\Bbb{\mathbb}
\def\bold{\mathbf}
\def\simtwohead{\,\, \hbox{\raise1pt\hbox{$^\sim$} \kern-13pt $\twoheadrightarrow \, $}}
\def\codim{\text{\rm codim\,}}
\def\stableto{\mapstochar \!\!\to}
\let\Sec=\S
\def\Z{\mathbb Z}
\def\Q{\mathbb Q}
\def\R{\mathbb R}
\def\E{\mathbb E}
\newcommand{\lra}{\longrightarrow}

\setcounter{tocdepth}{1}
\tableofcontents

\section{Introduction \label{sec:intro}}

The  interplay  between  dynamical systems and other branches of mathematics
is more than a century old. One of the early prototype results in differential topology, the Poincar\'e-Hopf theorem, equates the Euler 
characteristic of a compact smooth manifold with the enumeration of zeros of 
a generic vector field. In the 1930s, Marston Morse 
generalized this result, in what came to be known as the Morse inequalities, using {\it gradient dynamics}. 
A more recent interaction stems from {\it Hamiltonian dynamics.} The latter 
has inspired insights in the study of symplectic manifolds, 
enumerative geometry, string theory and algebraic topology. 

The scope of this paper is to erect yet another scaffold, one that will link
the fields of {\it stochastic dynamics,} enumerative combinatorics and
algebraic topology. Our investigation concerns the random motion of cellular
cycles in a finite CW complex.  Along the way, a higher dimensional analog of
electrical current will be defined as a homological observable for the random
process. The observable is subsequently used to relate the process to the
algebraic topology of the CW complex.  


To fix our ideas,
consider a finite, connected CW complex $X$ of dimension $d$. 
The first step of the program will be to associate to $X$, together with
auxiliary data, a continuous time Markov chain, which will hereafter be
referred to as a  {\it Markov CW chain}.  The latter will be a Markov process
in which a {\it state} is given by a $(k-1)$-cellular cycle in $X$ within a fixed
integer homology class. Such cycles are
constrained to evolve stochastically by jumping across  $k$-cells, i.e.,
``elementary homologies,'' where  each such jump adds the boundary of a
$k$-cell with a prescribed weight.  Our definition reduces in dimension one to
the notion of a time-dependent biased random walk on a graph, 
where an initial 0-cycle evolves by
jumping across 1-cells.  For now, we assume $k = d$. This represents no loss in
generality as we can replace $X$ by its $k$-skeleton, if necessary.



It is often convenient to represent the states and transitions of a Markov chain by
a {\it state diagram}. This is a directed topological graph with a vertex for
each state, where an edge corresponds to a transition between states.  The
state diagram for a Markov CW chain takes some care to define. To avoid tedium in this introduction, we will settle for an impressionistic description, referring the reader
to Definition \ref{defn:cycle_incidence} for the remaining details.  

As outlined above, the homology class of the
states of the system is fixed with respect to time evolution. 
Hence, we initially postulate that the allowed states of a Markov CW chain are the cellular $(d-1)$-cycles in $X$ over the integers that are 
homologous to a fixed initial cycle $z_0$. In other words, the set of allowed states is the coset 
\[
z_0+ B_{d-1}(X;\Bbb Z)\, ,
\]
where $B_{d-1}(X;\Bbb Z)$ is the abelian group of cellular $(d-1)$-boundaries.

Roughly, a transition from a state $z$ to a state $z'$ requires a choice
of $d$-cell $\alpha$ and a choice of $(d-1)$-cell $f$ such that
\[
  \langle z, f \rangle \neq 0 \quad \text{ and } \quad
  \langle \partial \alpha, f \rangle \neq 0 \, \, ,
\]
where $\langle -,- \rangle \in \Z$ denotes the incidence number. 
Furthermore, the target state $z'$ is obtained from $z$ by adding $\partial\alpha$, the boundary of $\alpha$, with
the proper coefficient $u \in \Bbb Z$:
\[
z' = z + u\partial \alpha
\]
(see Definition~\ref{defn:cycle_incidence} for the precise
statement). 

For processes on finite CW complexes $X$ with $d>1$, there will generally be an
infinite number of states, but only finitely many are directly accessible from
any given state by a single transition. Consequently, the state diagram is a
locally finite graph which may be globally infinite.  Furthermore, the states which are
inaccessible from $z_0$ by a sequence of elementary transitions are considered
as being decoupled from the process; for this reason we will omit
them.

The rate at which a cycle evolves is controlled by external data which we now
describe. Fix a real number $\beta >0$ which is to be interpreted as inverse temperature.
Label each $(d-1)$-cell $f$ by a real number $E_f$ and label each $d$-cell
$\alpha$ by a real number $W_{\alpha}$.  In the discrete-time case, a Markov CW
process is then described by labelling the above transition from $z$ to $z'$ with
the transition rate
\[
e^{\beta(E_f - W_\alpha)}\, .
\]
However, more interesting phenomena arise when considering
the continuous time case, in which 
the numbers $E_f$ and $W_\alpha$ are permitted
to vary in periodic, 1-parameter families.
Specifically,
the set of all such  labels $(E_f,W_\alpha)$ 
varying over all $(f,\alpha) \in X_{d-1}\times X_d$ 
will be  called the  {\it space of parameters}; denote it
by $\mM_X$. Then $\mM_X$ is a real vector space of dimension $|X_{d-1}| + |X_d|$, where
$X_k$ is the set of $k$-cells of $X$.  

A {\it driving protocol} is a smooth map
\[
\lambda\:\Bbb R \to \mM_X\, ,
\]
i.e.,  a 1-parameter family of labels $(E_f,W_\alpha)$
for every $(f,\alpha) \in X_{d-1}\times X_d$.

For a real number $\tau_D > 0$, we
say that $\lambda$ is {\it $\tau_D$-periodic} if 
$\lambda(t) = \lambda(t + \tau_D)$ for every $t  \in
\R$. In this case, $\lambda$ amounts to a choice of  pair
\[
(\tau_D,\gamma)\, ,
\]
where $\gamma\:[0,1] \to \mM_X $ is the smooth loop defined by
$\gamma(t) := \lambda(\tau_Dt)$. 

Given a $\tau_D$-periodic driving protocol,
there is a continuous-time Markov CW process that evolves 
by a (backward) Kolmogorov equation
\begin{equation}
\label{eqn:gen_master}
p'(t) \,\, =\,\,  -\tau_D \cal H(t) p(t) \, , \quad p(0) = z_0 \, ,
\end{equation}
where $p(t)$ is a one parameter family of
distributions (i.e., real-valued 0-cochains) on the state space.
The time-dependent operator $\cal H$ is 
the infinitesimal generator or transition rate matrix 
of the process (also known as the master operator or Fokker-Planck operator); 
it  is determined by the driving protocol
 $\gamma$  and inverse temperature $\beta >0$ (for the details see \S\ref{sec:process}). The  presence of the factor
 $\tau_D$ in equation \eqref{eqn:gen_master} is the result of time rescaling
$t\mapsto \tau_Dt$.

The differential
equation~\eqref{eqn:gen_master} is also known as the {\it master equation}. It has a formal solution $\varrho(t)$, which is
unique once an initial value $\varrho(0)$ is  specified.
Taking the weighted sum defined by $\varrho$, we obtain the  {\it expectation} 
\begin{equation} \label{eqn:intro-expect}
\Bbb E[\varrho]  = \sum_z \varrho(z)z\, ,
\end{equation}
(also known as the  {\it first moment}).
We will show that the expression \eqref{eqn:intro-expect}, which is typically infinite, always converges to a well-defined element of the group of real $(d-1)$-cycles $Z_{d-1}(X;\Bbb R)$.
 Applying 
the {\it biased coboundary operator} 
\[
\partial^\ast_{E,W} = e^{-\beta W}\partial^\ast e^{\beta E}\: C_{d-1}(X;\Bbb R) 
\to C_d(X;\Bbb R)
\]  
(cf.\ ~\eqref{eqn:biased_div} below; here $C_\ast(X;\Bbb R)$ denotes the real cellular chain complex of $X$) to $\Bbb E[\varrho]$  and integrating, we  
obtain a real cellular $d$-chain
\begin{equation} \label{eqn:avg-density}
Q := \int_0^1 \partial_{E,W}^\ast \Bbb E[\varrho]\, dt\, ,
\end{equation}
which can be viewed as the {\it average current} of the process; 
it depends on the triple $(\beta,\tau_D,\gamma)$.
The explanation for the terminology is that $\partial_{E,W}^\ast \Bbb E[\varrho] $ 
measures the flux of the expected value,
and the displayed integral is just the average value over time of this flux. 
 
When the driving time $\tau_D$ is sufficiently large, 
the expected value $\Bbb E[\varrho(t)]$  will be 1-periodic (cf.~Theorem \ref{bigthm:adiabatic}). Then 
$Q$ is a real $d$-cycle for large $\tau_D$ (by Theorem \ref{bigthm:expectation} below and the fundamental theorem of calculus). 
 Since $X$ has dimension $d$, there are no $d$-boundaries, so
 the group of real $d$-cycles coincides with the homology group  $H_d(X;\Bbb R)$.
Consequently, the average current will be a real homology class:
\begin{equation} \label{eqn:avg-density-homology}
Q  \in H_d(X;\Bbb R)\, ,
\end{equation}
which we view as a characteristic class for the Markov CW chain. 
Summarizing, we have associated an observable to any Markov CW chain that arises from periodic driving.
The current work investigates the properties of this homology class.

\subsection{Motivation and related work}

This paper is an  extension of the program
introduced in~\cite{Chernyak:algtop} to higher dimensions. A topological study
of continuous time random walks on graphs was performed there, and
an explicit result regarding the long time behavior of trajectories
was obtained. The results of that paper were proved using Kirchhoff's theorems on the  
flow of current in an electrical circuit.
While the motivations of the papers are similar, 
the generalization to higher dimensions introduces
formidable technicalities.

Other authors have considered generalizations of random walks to higher
dimensional simplicial complexes. Parzanchevski and Rosenthal~\cite{Parzan:SC}
define the {\it ($p$-lazy) $k$-walk} to be a Markov particle process on the set
of $k$-simplices of a simplicial complex. In their setup, a $k$-simplex
transitions to another $k$-simplex
through a co-face, and they relate this to the `up-down'
component of the $k$-Laplacian.  Mukherjee and Steenbergen~\cite{Mukherjee:RW}
consider a stochastic process where a $k$-simplex transitions to another $k$-simplex via a
face in the simplicial complex. The latter is related to the `down-up'
component of the Laplacian, and the two processes are dual to one another.
Rosenthal~\cite{Rosenthal:SBRW} also defined a {\it simplicial branching random
walk} (SBRW), which modifies the $k$-walk so that a $k$-simplex transitions to all
neighbors of an adjacent $(k+1)$-cell instead of a single neighbor.  Our notion of
Markov CW chain is closest to the SBRW of~\cite{Rosenthal:SBRW},
but is still distinct. 

The continuous time Markov chain considered in this paper is both
time-inhomogeneous and
not uniform (due to non-trivial values of $E$ and $W$). 
Even if we restrict to the embedded discrete time process, 
take trivial weights, and use a simplicial complex in which 
we collapse out the $(k-2)$-skeleton (so every $(k-1)$-cell is a cycle), 
our process is still distinct from the $k$-walk and the SBRW.
In the $k$-walk of~\cite{Parzan:SC} and~\cite{Mukherjee:RW},
a simplex transitions to a single neighbor, instead of all adjacent
neighbors as in the Markov CW chain.
In the language of~\cite{Rosenthal:SBRW}, the SBRW treats each of the $z_b=\langle z,b\rangle$ 
`particles' on a $k$-cell $b$ independently. This is in contrast to
the Markov CW chain, where all $z_b$ `particles' move 
together (see Figure~\ref{fig:intro_example}).

Our definition of a Markov CW chain is derived from the notion 
of a Langevin process in statistical mechanics. We 
imagine a smooth, compact, Riemannian manifold $(M,g)$ together with two
additional pieces of data.  The first is a Morse function $f\: M \to \R$ 
on a compact Riemannian manifold that satisfies 
the Morse-Smale transversality condition. Hence,
the Morse-Smale chain complex is defined.  The second is a stochastic
vector field on $M$ that possesses Gaussian and Markovian statistics. From 
a statistical mechanics perspective, the stochastic vector field arises 
from coupling our dynamical system to a `bath,' i.e., another dynamical
system with an enormous number of degrees of freedom.

These two ingredients can be used to define a stochastic flow
on $M$ (see e.g., \cite{Kunita:flows}).  An initial embedded $k$-dimensional submanifold
will then evolve according to the appropriate stochastic differential equation,
also known as a Langevin equation.  The setting of this manuscript is concerned
with the low-noise limit of these continuous processes under which the
stochastic motion becomes more deterministic and is restricted to the
associated Morse CW decomposition of $M$ given by the unstable manifolds of
$f$ (cf.\ \cite{Qin:CW}). In this way, the CW complexes considered here originate from
the Morse-Smale CW decompositions of smooth manifolds (this is the basis
for Hypothesis~\ref{hypo:decomp}).  It is worth noting that the smooth
setting is what distinguishes our process from those already  appearing in the literature. The embedded
submanifold is a $k$-cycle in bordism homology, forcing the state space 
to consist of $k$-cycles instead of $k$-cells. 
If the stochastic diffeomorphism pushes a
portion of the $k$-cycle $z$ over a $(k+1)$-cell and onto all adjacent $k$-cells,
it must do so uniformly. That is, all $z_b = \langle z, b \rangle$ `particles' on a $k$-cell 
$b$ move together, not independently.

\begin{figure}[h]
  \centering
  \begin{tikzpicture}
    \draw (0,0) -- (0,1) -- (1,2) -- (2,1) -- (2,0) -- (0,0);
    \draw [thick,red] (-0.3,-0.3) -- (0,0);
    \draw [thick,red] (0,0) -- (0,1) node [midway,left] {\tiny $j$};
    \draw (0,0) -- (2,0) node [midway,below] {\tiny $0$};
    \draw (2,0) -- (2,1) node [midway,right] {\tiny $0$};
    \draw (2,1) -- (1,2) node [midway,above] {\tiny $0$};
    \draw (1,2) -- (0,1) node [midway,above] {\tiny $0$};
    \draw[thick,red] (0,1) -- (-0.3, 1.3);
    \node [label={[red,rotate=140]$\ldots$}] at (-0.45,1.5) {};
    \node [label={[red,rotate=-140]$\ldots$}] at (-0.65,-0.55) {};
  \end{tikzpicture}
  \raisebox{3.5em}{$\Rightarrow$}
  \hspace{-1em}
  \begin{tikzpicture}
    \draw (0,0) -- (0,1) -- (1,2) -- (2,1) -- (2,0) -- (0,0);
    \draw[thick,red] (-0.3,-0.3) -- (0,0);
    \draw[thick,red] (0,1) -- (-0.3,1.3);
    \node [label={[red,rotate=140]$\ldots$}] at (-0.45,1.5) {};
    \node [label={[red,rotate=-140]$\ldots$}] at (-0.65,-0.55) {};
    \draw[thick,red] plot [smooth, tension=1.2] coordinates{(0,0) (0.7,0.6) (0,1)};
    \draw[thick,red] plot [smooth,tension=0.9] coordinates{(0,0) (1.1,0.2) (1.5,0.6) (1.1,1.3) (0,1)};
    \node [label={[red,scale=0.8,rotate=20]$\rightarrow$}] at (0.4,0.2) {};
    \node [label={[red,scale=0.8,rotate=15]$\rightarrow$}] at (1.2,0.4) {};
    \node [label={[red,scale=0.8,rotate=90]$\rightarrow$}] at (1.15,1.5) {};
    \node [label={[red,scale=0.8,rotate=-15]$\rightarrow$}] at (1.7,0.2) {};
  \end{tikzpicture}
  \raisebox{3.5em}{$\Rightarrow$}
  \hspace{-1em}
  \begin{tikzpicture}
    \draw (0,1) -- (1,2) -- (2,1) -- (2,0) -- (0,0);
    \draw (0,0) -- (0,1) node [midway,left] {\tiny $0$};
    \draw [thick,red] (-0.3,-0.3) -- (0,0);
    \draw[thick,red] (0,0) -- (2,0) node [midway,below] {\tiny $j$};
    \draw[thick,red] (2,0) -- (2,1) node [midway,right] {\tiny $j$};
    \draw[thick,red] (2,1) -- (1,2) node [midway,above] {\tiny $j$};
    \draw[thick,red] (1,2) -- (0,1) node [midway,above] {\tiny $j$};
    \draw[thick,red] (0,1) -- (-0.3, 1.3);
    \node [label={[red,rotate=140]$\ldots$}] at (-0.45,1.5) {};
    \node [label={[red,rotate=-140]$\ldots$}] at (-0.65,-0.55) {};
  \end{tikzpicture}
  \caption{The motivation for an elementary transition of the Markov CW chain. 
    The initial cycle has incidence $j$ with one face of a 2-cell. The stochastic vector field
  pushes the cycle off the face and across the entire 2-cell. Only the first and last pictures
  take place in the process on the CW complex; the intermediate figure is the smooth manifold picture
  which motivates our definition.}
  \label{fig:intro_example}
\end{figure}
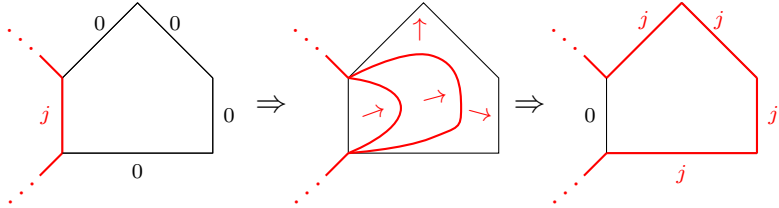

\subsection{Statement of Results}
\label{sec:statement_of_results}

For the motion of points on both graphs and smooth
manifolds, the Fokker-Planck operator takes the form of a
biased Laplacian~\cite{Gardiner:SM, Ikeda:SDE}. Suitably understood, the master operator will also be a biased Laplacian.
However, there are analytical difficulties in working
with the master equation directly, 
as the space of distributions on the set of states is typically  
infinite dimensional.
Fortunately, the expectation
of the formal solution of the master equation
 also satisfies a Kolmogorov-type dynamical equation in which 
 the dynamical operator acts
 on the finite dimensional vector space of cellular $(d-1)$-chains
 $C_{d-1}(X;\Bbb R)$, as we will now explain.

The {\it (reduced) biased Laplacian} $H\: C_{d-1}(X;\Bbb R) \to C_{d-1}(X;\Bbb R)$
is the operator given by
\[
\partial \partial_{E,W}^\ast = \partial e^{-\beta W}\partial^\ast e^{\beta E}\, .
\]
Note that $H$ is defined in terms of $\gamma$ and $\beta$; in particular $H$ is time-dependent.
The {\it dynamical equation} is given by
\begin{equation}
  \dot q  = -\tau_DH q\, .
\label{eqn:dynamical_eqn_gen}
\end{equation}

\begin{bigthm}[Expectation Dynamics]
  \label{bigthm:expectation} Let $\varrho(t)$ be the formal
solution of the master equation \eqref{eqn:gen_master}  
with initial value $\varrho(0) = z_0$. Then its expectation
  \[
    \rho(t) := \E[\varrho(t)]
  \]
is the unique solution to the dynamical equation \eqref{eqn:dynamical_eqn_gen}
with respect to the initial condition $\rho(0) = z_0$.
\end{bigthm}

Theorem   \ref{bigthm:expectation} is the cornerstone of our investigation:
it relates the evolution of a state of the process
 to the evolution of its first moment. The latter is more directly related to the
 topology of the CW complex $X$.
 
The next step of the program is to analyze $\rho$ under two limits on the process.  The
first of these is the {\it adiabatic limit}, in which $\tau_D \to \infty$.  The term
``adiabatic'' appreciates the sufficiently slow variation of the parameters.


\begin{bigthm}[Adiabatic Theorem] \label{bigthm:adiabatic}
There exists a positive real number $\tau_0 = \tau_0(\beta, \gamma)$ 
such that for all $\tau_D > \tau_0$, a
1-periodic solution $\rho_{\tau_D}$ of the dynamical equation
\eqref{eqn:dynamical_eqn_gen}
 exists and is unique. Furthermore,
 \[
   \lim_{\tau_D \to \infty} \rho_{\tau_D} = \rho^B\, ,
 \]
 where $\rho^B  = \rho^B(\gamma,\beta)$ is the Boltzmann distribution at $[z_0]\in H_{d-1}(X;\Bbb R)$ (cf.~Definition~\ref{defn:Boltzmann_dist}).
\end{bigthm}

In particular, Theorem \ref{bigthm:adiabatic} shows that the long time behavior of the process 
is no longer dynamical in nature. Furthermore, the limiting cycle is
given by the Boltzmann distribution~\cite{CCK:Boltzmann}, which can be interpreted as the
unique harmonic form on this class (see Theorem~\ref{thm:Boltzmann}). 
This cycle is a weighted
average over the cycles representing the homology class, so Theorem \ref{bigthm:adiabatic} 
is a kind of ergodic theorem for the expectation of our
Markov process.  

\begin{rem} We reiterate that Theorem~\ref{bigthm:adiabatic}
refers to first moment of the process, whereas the actual process typically
has no limiting distribution (in fact, it can blow up in finite time).
\end{rem}

Recall that the average current $Q$ is defined in 
terms of the parameters $(\tau_D,\gamma,\beta)$. In what follows set $Q = Q_{\tau_D}$
to emphasize its dependence on $\tau_D$.
 Set
\[
Q^B := \lim_{\tau_D \to \infty} Q_{\tau_D} \, .
\]
Theorem \ref{bigthm:adiabatic}  implies that $Q^B$ is well-defined and
depends only on Boltzmann
distribution $\rho^B$. 

The second limit we are interested in is the {\it low temperature limit}, or
low noise limit, under which $\beta \to \infty$.  The main result of this
paper is stated in the low temperature, adiabatic limit. In what follows we
write $Q^B_\beta$ to indicate the dependence of $Q^B$ on $\beta$.

\begin{bigthm}[Quantization]
  \label{bigthm:quant} 
  Assume $X$ is a connected finite CW complex of dimension $d$. For a ``good''
  periodic driving protocol $(\tau_D,\gamma)$, the low temperature, adiabatic
  limit of the average current is well-defined and fractionally quantizes, i.e., there is a
  positive integer $\delta$ such that
  \[
    \lim_{\beta \to \infty} Q^B_\beta
    \in H_d(X;\Z[\tfrac{1}{\delta}]) \subset H_d(X;\R) \, ,
  \]
 Morever, the $\delta$ is a combinatorial invariant of $X$ (cf.\ Theorem~\ref{thm:quantization}).
\end{bigthm}

The term ``good'' that appears in Theorem \ref{bigthm:quant} is a generic property: it refers to those driving protocols whose image lies in a suitable open and dense
topological subspace of $\cal M_X$. More precisely, the subspace we take, denoted by
$\breve {\cal M}_X$, is the subset of pairs
$(E,W)$ in which either $E\: X_{d-1} \to \Bbb R$ or $W\: X_d \to \Bbb R$ is one-to-one.

If $d=1$, then $\delta=1$ and Theorem \ref{bigthm:quant}
recovers a version of the integral quantization result of~\cite[thm.~A]{Chernyak:algtop}. In higher dimensions, the appearance of torsion phenomena in the integral homology of
$X$ is partly responsible for the inversion of the number $\delta$.
We consider Theorem \ref{bigthm:quant} to be the main result of this manuscript.

\subsection{An example} For $d \ge 2$, let $c\: S^{d-1}\to S^{d-1} \vee S^{d-1}$
 be the $(d-2)$-fold suspension of the map $S^1 \to S^1 \vee S^1$ which is given by the loop multiplication $xy^{-1}$, where $x,y$ denote the two inclusions of $S^1$ into $S^1 \vee S^1$.
 
Let 
\[
X = (S^{d-1} \vee S^{d-1}) \cup (D^d \amalg D^d)
\]
 be the CW complex of dimension $d$  given by attaching two $d$-cells
to $S^{d-1} \vee S^{d-1}$, each one using the map $c$. Then 
\[
H_{d}(X;\Bbb Z) \cong H_{d-1}(X;\Bbb Z) \cong \Bbb Z\, .
\]  
Denote the $(d-1)$-cells of $X$ by $f_1,f_2$ and the $d$-cells by $e_1,e_2$. We take $x\in H_{d-1}(X;\Bbb Z)$ to be the generator defined by $f_1$.

Let $W_1\: [0,1] \to [-1,1]$ be any smooth function which vanishes for $t \in \{0,1/2,1\}$
and which satisfies $W_1(t) < 0$ for $t\in (0,1/2)$ and $W_1(t) > 0$ for $t\in (1/2,1)$.
We take $W_1$ to be a one-parameter family of weights for the $d$-cell $e_1$.
Set $W_2(t) := -W_1(t)$, providing a
family of weights for the $d$-cell $e_2$. Let $E_1\:[0,1] \to [-1,1]$ be any smooth function such that
$E_1(1/2) > 1$ and $E_1(0) = E_1(1) = -1$. Set $E_2(t) = -E_1(t)$. 
Then $\gamma(t):= (E_\bullet(t),W_\bullet(t))$ defines a good 1-periodic continuous driving protocol on $X$.
If we additionally assume  $E_i'(0) = E_i'(1), W_i'(0) = W_i'(1)$ for $i=1,2$, then $\gamma$ will be smooth. 
  
\begin{bigthm}\label{bigthm:ex}  With respect to the above choices, 
the average current in the low temperature, adiabiatic limit coincides with the generator
of $H_d(X;\Bbb Z)$ given by the cycle $e_1-e_2$.
\end{bigthm} 

Integer coefficients occur in Theorem \ref{bigthm:ex},
since in this example $\delta=1$, where $\delta$ is as in Theorem \ref{bigthm:quant}.

\begin{out} Section \ref{sec:prelim} is about language. In section
\ref{sec:combinatorial_struct} we review some material on spanning trees and spanning co-trees that appears in our earlier papers \cite{CCK:Kirchhoff}, \cite{CCK:Boltzmann}. In section \ref{sec:process} we define the Markov CW chain, derive its basic properties and then give a proof of Theorem \ref{bigthm:expectation}.  The proof of the Adiabatic Theorem appears in section \ref{sec:adiabatic}. In section \ref{sec:temp}, we determine the low temperature limit of the time-dependent Boltzmann distribution. Section \ref{sec:quantize} contains the proof of the Quantization Theorem. Section \ref{sec:example} validates the example
(Theorem \ref{bigthm:ex}).
\end{out}

\begin{acks}  The results of this paper are a distillation
of the first author's Ph.~\!\!D.~dissertation.
During the period when this paper was being written,
the third author received generous support
from the Eurias foundation and the Israel Institute for Advanced Studies.
This material is based upon work supported by the National Science
Foundation Grant CHE-1111350 and the Simons Foundation Collaboration
Grant 317496.
\end{acks}

\section{Preliminaries \label{sec:prelim}}

\subsection{Notation}
Many of the functions appearing in this paper depend on several variables. To avoid clutter we typically avoid displaying the function arguments. However, when a particular
variable is to be emphasized we display it. For example, if $f(t,x,y,z)$
is a function of four variables, we typically write it as $f$. When we wish to emphasize some of the variables, say $t,z$, we write $f = f_{t,z}$ or $f = f(t,z)$.
If $f$ is differentiable in the $t$-variable, where $t$ is viewed as time, we write $\dot f$ for its time derivative (i.e., $\tfrac{\partial f}{\partial t}$).

\subsection{CW Complexes}
Let $X$ be a finite CW complex of fixed dimension $d \geq 1$.
We denote the $k$-skeleton of $X$ by $X^{(k)}$, and the set of $k$-cells
by $X_k$. We are primarily interested in the top dimensions of $X$ 
($d$ and $d-1$), but our results
hold for any intermediate dimension $k$ by truncation to $X^{(k)}$.

Recall that the CW structure of $X$ 
is specified inductively by attaching cells of increasing dimension.
The $k$-skeleton is formed from the $(k-1)$-skeleton by means of 
attaching maps
\[
  S^{k-1}_{\alpha} @> \varphi_{\alpha} >> X^{(k-1)}
\]
where $\alpha$ indexes the set of $k$-cells to be attached. Then 
\[
  X^{(k)} = X^{(k-1)} \cup \coprod_{\alpha} D_{\alpha}^k \, ,
\]
where the disjoint union is amalgamated along the attaching maps.

\subsection{The cellular chain complex}  For a commutative ring $A$, let 
\[
C_k := C_k(X;A)
\] denote the free $A$-module
with basis $X_k$. In this section, $A =\Bbb R$ is the field of real numbers.
For now, equip $C_k$ with the standard inner product $\langle{-},{-}\rangle$
by declaring $X_k$ to be an orthonormal basis. 
Recall that $C_\ast$ forms a chain complex of abelian groups (vector spaces when $A = \Bbb R$), in which
the effect of the boundary operator $\partial\: C_k \to C_{k-1}$
on a $k$-cell $\alpha$ is
\[
  \partial \alpha = \sum_{j \in X_{k-1}}
    b_{\alpha, j} j\, ,
\]
where $b_{\alpha, j} := \langle \partial \alpha, j \rangle$
is the incidence number of $\alpha$ and $j$; this is a finite sum. 
The incidence number
can be explicitly described by means of the attaching maps: 
$b_{\alpha, j}$ is the degree of the composite
\[
 S^{k-1}_{\alpha} @>\varphi_{\alpha}>> X^{(k-1)} @>>> X^{(k-1)}/X^{(k-2)}
 \cong \bigvee_i S^{k-1}_i @>>> S^{k-1}_j,
\]
where the last map is given by projection onto the wedge summand 
corresponding to cell $j$.

We will  assume $X$ comes equipped with the following auxiliary structure:

\begin{hypo_no}
  \label{hypo:decomp} Let $k$ be an integer satisfying $1\le k \le d$.
  For every $k$-cell $\alpha$ and $(k-1)$-cell $j$, we fix a
  choice of finite set $X(\alpha,j)$ such that the cellular boundary map
  $\partial \: C_k \to C_{k-1}$  admits an additional decomposition:
  \[
    b_{\alpha, j} = \sum_{\varepsilon_{\alpha,j} \in X(\alpha,j)} 
    (-1)^{\chi(\varepsilon_{\alpha,j})}\, ,
  \]
  where $\chi \in \{0,1\}$.
\end{hypo_no}

\begin{rems}(1).~Clearly, one can always make choices so that the hypothesis is satisfied. However,
in the main cases of interest the decomposition comes for free. 

In fact, the hypothesis is inspired by properties of the boundary map in the Morse-Smale complex of a 
Morse function $f : M \to \R$ on a compact Riemannian manifold
$M$ satisfying the Morse-Smale transversality condition. 
In the Morse-Smale case, we take
$X(\alpha,j)$ to the (finite) set of flow lines between the corresponding critical points of index $k$ and $k-1$,
$b_{\alpha,j}$ is a signed sum over the flow lines, 
with $(-1)^{\chi(\varepsilon)}$ the sign of flow line $\varepsilon$.
\smallskip

\noindent (2).~There are other cases of interest in which the hypothesis holds without additional choices:
if $X$ is a regular CW complex, connected polyhedron, or
simplicial complex, then the hypothesis holds with $|X(\alpha,j)| = 1$ for every $\alpha$ and $j$, and
therefore, $b_{\alpha,j} = \pm 1$.
\end{rems}

The {\it coboundary operator} $\partial^*\: C_{k-1} \to C_k$ is the formal
adjoint to the boundary operator in the standard inner products. 
Explicitly,
\[
  \partial^*j = \sum_{\alpha \in X_d}
   b^*_{j,\alpha} \alpha \, ,
\]
for any $(k-1)$-cell $j$, where $b^*_{j, \alpha} := b_{\alpha, j}$.

%
%
%
%
%

We write $H_\ast(X;A)$ for the {\it cellular homology} of $X$ with coefficients in $A$. That is
\[
H_k(X;A) \, :=\, 
Z_k/B_k\, ,
\]
where $Z_k$, the group of $k$-cycles, is the kernel of 
the homomorphism $\partial \: C_k \to C_{k-1}$, and
$B_k$, the group of $k$-boundaries, is the image of $\partial \: C_{k+1} \to C_k$.
The quotient is well-defined since $\partial \circ \partial = 0$.

\subsection{Weight systems}
For the remainder of the paper, we assume that 
\[
k = d := \dim X\, .
\]
Fix a real number $\beta>0$, known as {\it inverse temperature}. 

\begin{defn}
  A system of {\it weights} for $X$ consists of functions
  \[
    E\:X_{d-1} \to \R \quad \text{ and }  \quad
    W\: X_d \to \R \, .
  \]
  We write $E_j := E(j)$ for a $(d-1)$-cell $j$ and 
  $W_{\alpha} := W(\alpha)$ for a
  $d$-cell $\alpha$.
\end{defn}

Fixing the weights for the moment and a real number $\beta >0$, define (diagonal) operators
\begin{align*}
  e^{\beta E}\: C_{d-1}(X;\R) &\to C_{d-1}(X;\R) \quad \text{ and } \quad
  e^{\beta W}\: C_d(X;\R) \to C_d(X;\R)\, ,
\end{align*}
by
\begin{align}\label{eqn:op_chain_cx}
  j &\mapsto e^{\beta E_j} j  \quad \text{ and } \quad
  \alpha \mapsto e^{\beta W_{\alpha}} \alpha\, ,
\end{align}
for $j \in X_{d-1}$ and  $\alpha \in X_d$. 
We use these  to equip $C_d(X;\R)$ and 
$C_{d-1}(X;\R)$ with {\it modified inner products:} 
for $i,j \in X_{d-1}$ and $\alpha,\gamma\in X_{d}$, set
\begin{gather}\label{eqn:modified_ip}
  \langle i,j \rangle_E := e^{\beta E_i} \delta_{ij} 
  \quad \text{ and } \quad \langle \alpha, \gamma \rangle_W := e^{\beta W_{\alpha}} 
  \delta_{\alpha \gamma} \, ,
\end{gather}
where in this case
$\delta$ denotes Kronecker delta. The modified inner products
are then given by extending these formulas bilinearly.

If we define the formal adjoint of $\partial$ using the modified
inner products, we obtain
the {\it biased coboundary operator;} explicitly, 
\begin{gather}\label{eqn:biased_div}
  \partial_{E,W}^* \, := \, e^{-\beta W} \partial^* e^{\beta E} \, ,
\end{gather}
where $\partial^*$, the standard coboundary operator, is the formal adjoint
 with respect to the standard inner
products.

\section{Combinatorial structures \label{sec:combinatorial_struct}}

We briefly recall
the properties of spanning trees and spanning co-trees in this section.
We do not present any 
new results in this section and we refer the reader to~\cite{CCK:Kirchhoff} 
and~\cite{CCK:Boltzmann} for a more complete treatment. 
For a finite complex $Y$, let $\beta_k(Y)$ be the $k$-th betti number, i.e., the rank of $H_k(Y;\Bbb Q)$.

\subsection{Spanning trees}

\begin{defn} Assume $\dim X = d \ge 1$ and let $1\le k \le d$ .
A {\it k-spanning tree} for $X$ is a subcomplex $i: T \subset X$
such that
\begin{itemize}
  \item $H_k(T;\Z) \cong 0$, and
  \item $\beta_{k-1}(T) =\beta_{k-1}(X)$, and 
  \item $X^{(k-1)} \subset T \subset X^{(k)}$.
\end{itemize}
When $k=d$, we simplify the terminology to {\it spanning tree}.
\end{defn}

\begin{rem} If $d=k=1$ then the above  coincides with the usual
notion of spanning tree in a connected graph.
\end{rem}

\begin{defn}
  A $k$-cell $b \in X_k$ is said to be {\it essential} if there exists
  a $k$-cycle $z \in Z_k(X;\R)$ such that $\langle z,b \rangle \neq 0$.
\end{defn}

Removing an essential $k$-cell from $X^{(k)}$ results in a complex in which
$\beta_{k}$ decreases by one and $\beta_{k-1}$ is 
fixed~\cite[Lemma 2.2]{CCK:Kirchhoff}. Every $k$-spanning tree can therefore
be constructed by iteratively removing essential $k$-cells from $X^{(k)}$.

\begin{defn} Let $k=d$.
  For a spanning tree $T$, define a linear transformation
  \begin{gather}
    \label{eqn:K-T}
    \varsigma_T\: B_{d-1}(X;\Q) \to C_{d}(T;\Q)
  \end{gather}
  as follows: $\varsigma_T(b)$ is the unique $d$-chain in $T$ so that
  $\partial \varsigma_T(b) = b$.
\end{defn}

The $d$-chain $\varsigma_T(b)$ exists since $B_{d-1}(T;\Q) = B_{d-1}(X;\Q)$ for every spanning
tree $T$. The chain is unique since the difference of any 
two distinct $d$-chains with boundary $b$ would give rise to a 
non-trivial $d$-cycle in $T$, for 
which there are none.


\begin{defn} For a given system of weights $(E,W)$ on $X$,
 the {\it weight} of a spanning tree $T$ is the positive real number
  \[
    w_T := \theta_T^2 \prod_{\alpha \in T_d} e^{-\beta W_\alpha} \, \, ,
  \]
  where $\theta_T$ denotes the order of the torsion subgroup of $H_{d-1}(T;\Z)$.
\end{defn}
 
 \begin{thm}[{cf.~\cite[thm.~A]{CCK:Kirchhoff}}]
   \label{thm:Kirchhoff_network}
  With respect to the modified inner product $\langle - , - \rangle_W$, 
  an orthogonal splitting to the boundary operator $\partial\: C_d(X;\R) 
  \to B_{d-1}(X;\R)$ is given by
  \begin{equation}\label{eqn:K_op}
\cal A:= \tfrac{1}{\Delta} \sum_T w_T \varsigma_T\, ,
 \end{equation}
  where the sum is over all spanning trees, and $\Delta = \sum_T w_T$.
\end{thm}

\begin{rem} \label{rem:ortho-section-bdy}
The map $\cal A$ is a orthogonal splitting of $\partial$ in the
short exact sequence
\[
  0 \lra Z_d(X;\R) \stackrel{i}{\lra} C_d(X;\R) \stackrel{\partial}{\lra} B_{d-1}(X;\R) \lra 0 \, ,
\]
with respect to the modified inner product $\langle - , - \rangle_W$. It follows
that 
\[
I- \cal A\partial\:  C_{d}(X;\R) \to Z_d(X;\R)
\]
gives the orthogonal projection
of $i$. The latter operator was 
constructed explicitly and studied in~\cite{CCK:Kirchhoff}.
\end{rem}

\subsection{Spanning co-trees}

\begin{defn} Assume $\dim X = d \ge 1$. Fix an integer $k$ with $0 \le k \le d$.
  A {\it k-spanning co-tree} for $X$ is a 
  subcomplex $j:L \subset X$ such that
  \begin{itemize}
    \item $j_*: H_{k}(L;\Q) \to H_{k}(X;\Q)$ is an isomorphism,
    \item $\beta_{k-1}(L) = \beta_{k-1}(X)$, and 
    \item $X^{(k-1)} \subset L \subset X^{(k)}$.
  \end{itemize} 
  When $k=d-1$ we shorten the terminology to {\it spanning co-tree}.
\end{defn}

\begin{rem} Similar to $k$-spanning trees, $k$-spanning co-trees are shown to exist by
removing certain $k$-cells from $X^{(k)}$. 

Note that a 0-spanning co-tree is just a 0-cell of $X$. There is
only one $d$-spanning co-tree given by $X$. 
\end{rem}

We now restrict to the case $k=d-1$.
Since a spanning co-tree $L$
has no $d$-cells, the relative homology group $H_{d-1}(X,L;\Q)$ 
is trivial. It follows that $H_{d-1}(X,L;\Z)$ is finite; let $a_L$
denote its order. Note that the composite
\[
  \phi_L\: Z_{d-1}(L;\Z) @> \cong >> H_{d-1}(L;\Z) \to H_{d-1}(X;\Z).
\]
is a rational isomorphism since $L$ has no $d$-cells. 

\begin{defn} With $L$ as above, let
 $\psi_L\: H_{d-1}(X;\Q) \to Z_{d-1}(X;\Q)$ denote the 
  composite
  \[
 H_{d-1}(X;\Q) @> (\phi_L \otimes \Q)^{-1}>\cong> Z_{d-1}(L;\Q) @> j_{*} >> Z_{d-1}(X;\Q) \, .
  \]
\end{defn}

\begin{defn} For a given system of weights $(E,W)$ on $X$,
  the {\it weight} of a spanning co-tree $L$ is the positive real number 
  \[
    b_L = a_L^2 \prod_{b\in L_{d-1}} e^{-\beta E_b} \,\, .
  \]
\end{defn}

\begin{thm}[{\cite[thm.~A]{CCK:Boltzmann}}]
  \label{thm:Boltzmann}
  With respect to the modified inner product $\langle - , 
  - \rangle_E$, the orthogonal splitting to the quotient homomorphism
  $Z_{d-1}(X;\R) \to H_{d-1}(X;\R)$ is given by
  \begin{gather}
\label{eqn:rhoB}
    \rho^B(E) = \rho^B = \tfrac{1}{\nabla} \sum_L b_L \psi_L\, ,
  \end{gather}
  where the sum is over all spanning co-trees $L$, and $\nabla = \sum_L b_L$.
\end{thm}

\begin{defn}[{cf.~\cite[defn.~1.12]{CCK:Boltzmann}}] \label{defn:Boltzmann_dist} 
Let  $x \in H_{d-1}(X;\Z)$ be an integer homology class.
The {\it Boltzmann distribution at $x$}
is the real $(d-1)$-cycle
\[
\rho^B(x) := \tfrac{1}{\nabla}\sum_L b_L \psi_L(\bar x) \in Z_{d-1}(X;\R)\, ,
\]
where $\bar x \in H_{d-1}(X;\Q)$ is the image of $x$ under the homomorphism
$H_{d-1}(X;\Z) \to H_{d-1}(X;\Q)$.
\end{defn}

\begin{rem}
  \label{rem:D-coeffs}
  For a spanning tree $T$, let $A \subset \Bbb Q$ be a ring
  in which $\theta_T$ is a unit. An elementary  diagram chase 
involving the long exact sequence in 
  homology of the pair $(X,T)$ implies that 
  the linear transformation $\varsigma_T$ uniquely lifts to a homomorphism
 \[
B_{d-1}(X;A) \to C_{d}(T;A)\, .
 \]
  
Similarly, for any spanning co-tree $L$, 
 if the $a_L$ is a unit in $A$, then 
  $\psi_L$ uniquely lifts to a homomorphism
  \[
H_{d-1}(X;A) \to Z_{d-1}(X;A)\, .
  \]
  \end{rem}


\begin{rem}
  \label{rem:harmonic_form}
  The Boltzmann distribution is the unique `harmonic form' on $X$ as 
  specified by combinatorial Hodge theory (see~\cite{CCK:Boltzmann}).
  Remark~\ref{rem:D-coeffs} specifies the minimal 
  coefficients under which the harmonic form of a 
  homology class will exist.
\end{rem}

\begin{ex}
\label{ex:torus}

Let $X$ denote the torus with CW structure given by four $0$-cells,
eight $1$-cells, and four $2$-cells, shown in Figure~\ref{fig:torus}. 
We make the usual identifications
of opposite sides in this picture, although this is not shown explicitly. Instead,
the displayed arrows label a chosen orientation.

This complex has four $2$-spanning trees, given by removing 
any single 2-cell. There are thirty-two 1-spanning trees, 
obtained by subtracting the 24 loops of $X^{(1)}$ from 
the 56 possible choices of 3 edges.

On the other hand, 
there are thirty-two 1-spanning co-trees, and four 0-spanning
co-trees (cf.\ Figure \ref{fig:trees_cotrees}). These statements can be obtained by careful
enumeration or by using Theorem~\cite[Corollary D]{CCK:Boltzmann}.
\end{ex}

\begin{center}
  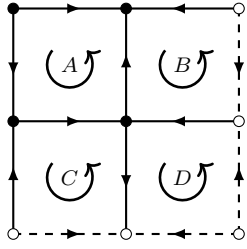
\begin{figure}
\begin{tikzpicture}[scale=1.5]
\begin{scope}[thick,decoration={
    markings,
    mark=at position 0.6 with {\arrow{latex}}}
    ] 
\draw [postaction={decorate}] (0,0) -- (0,1);
\draw [postaction={decorate}] (0,2) -- (0,1);
\draw [postaction={decorate}] (1,1) -- (1,0);
\draw [postaction={decorate}] (1,1) -- (1,2);
\draw [postaction={decorate}] (0,2) -- (1,2);
\draw [postaction={decorate}] (2,2) -- (1,2);
\draw [postaction={decorate}] (0,1) -- (1,1);
\draw [postaction={decorate}] (2,1) -- (1,1);
\draw [postaction={decorate},dashed] (2,0) -- (2,1);
\draw [postaction={decorate},dashed] (2,2) -- (2,1);
\draw [postaction={decorate},dashed] (0,0) -- (1,0);
\draw [postaction={decorate},dashed] (2,0) -- (1,0);
\draw [->,line width = 1pt] (.5,.5) ++(140:2mm) arc (-220:40:2mm);
\draw [->,line width = 1pt] (.5,1.5) ++(140:2mm) arc (-220:40:2mm);
\draw [->,line width = 1pt] (1.5,.5) ++(140:2mm) arc (-220:40:2mm);
\draw [->,line width = 1pt] (1.5,1.5) ++(140:2mm) arc (-220:40:2mm);
\end{scope}
\node[align=left] at (.5,1.5)%
{\tiny $A$};
\node[align=left] at (1.5,1.5)%
{\tiny $B$};
\node[align=left] at (.5,.5)%
{\tiny $C$};
\node[align=left] at (1.5,.5)%
{\tiny $D$};
\node at (0,0) [circle,draw=black,fill=white,inner sep=0.5mm] {};
\node at (0,1) [circle,draw=black, fill=black, inner sep=0.5mm] {};
\node at (0,2) [circle,draw=black, fill=black, inner sep=0.5mm] {};
\node at (1,0) [circle,draw=black, fill=white, inner sep=0.5mm] {};
\node at (1,1) [circle,draw=black, fill=black, inner sep=0.5mm] {};
\node at (1,2) [circle,draw=black, fill=black, inner sep=0.5mm] {};
\node at (2,0) [circle,draw=black, fill=white, inner sep=0.5mm] {};
\node at (2,1) [circle,draw=black, fill=white, inner sep=0.5mm] {};
\node at (2,2) [circle,draw=black, fill=white, inner sep=0.5mm] {};
\end{tikzpicture}
\caption{A torus with the usual identifications of opposite sides made.
The arrows denote a choice of orientation.}
\label{fig:torus}
\end{figure}
\end{center}

\begin{figure}[h]
  \centering
  \begin{subfigure}{\textwidth}
    \centering
\begin{tikzpicture}[scale=1] 
\draw [thick, dotted] (0,0) -- (0,1) ;
\draw [ultra thick] (0,1) -- (0,2) ;
\draw [thick,dotted] (1,0) -- (1,1) ;
\draw [ultra thick] (1,1) -- (1,2) ;
\draw [ultra thick] (0,1) -- (1,1) ;
\draw [thick,dotted] (1,1) -- (2,1) ;
\draw [thick, dotted] (0,2) -- (1,2) ;
\draw [thick,dotted] (1,2) -- (2,2) ;
\draw [thick, dashed] (2,0) -- (2,1) -- (2,2);
\draw [thick, dashed] (0,0) -- (1,0) -- (2,0);
\node at (0,0) [circle,draw=black,fill=white,inner sep=0.5mm] {};
\node at (0,1) [circle,draw=black, fill=black, inner sep=0.5mm] {};
\node at (0,2) [circle,draw=black, fill=black, inner sep=0.5mm] {};
\node at (1,0) [circle,draw=black, fill=white, inner sep=0.5mm] {};
\node at (1,1) [circle,draw=black, fill=black, inner sep=0.5mm] {};
\node at (1,2) [circle,draw=black, fill=black, inner sep=0.5mm] {};
\node at (2,0) [circle,draw=black, fill=white, inner sep=0.5mm] {};
\node at (2,1) [circle,draw=black, fill=white, inner sep=0.5mm] {};
\node at (2,2) [circle,draw=black, fill=white, inner sep=0.5mm] {};
\end{tikzpicture}
\begin{tikzpicture}[scale=1] 
\draw [thick,dotted] (0,0) -- (0,1) ;
\draw [ultra thick] (0,1) -- (0,2) ;
\draw [thick,dotted] (1,0) -- (1,1) ;
\draw [thick,dotted] (1,1) -- (1,2) ;
\draw [thick,dotted] (0,1) -- (1,1) ;
\draw [ultra thick] (1,1) -- (2,1) ;
\draw [ultra thick] (0,2) -- (1,2) ;
\draw [thick,dotted] (1,2) -- (2,2) ;
\draw [thick, dashed] (2,0) -- (2,1) -- (2,2);
\draw [thick, dashed] (0,0) -- (1,0) -- (2,0);
\node at (0,0) [circle,draw=black,fill=white,inner sep=0.5mm] {};
\node at (0,1) [circle,draw=black, fill=black, inner sep=0.5mm] {};
\node at (0,2) [circle,draw=black, fill=black, inner sep=0.5mm] {};
\node at (1,0) [circle,draw=black, fill=white, inner sep=0.5mm] {};
\node at (1,1) [circle,draw=black, fill=black, inner sep=0.5mm] {};
\node at (1,2) [circle,draw=black, fill=black, inner sep=0.5mm] {};
\node at (2,0) [circle,draw=black, fill=white, inner sep=0.5mm] {};
\node at (2,1) [circle,draw=black, fill=white, inner sep=0.5mm] {};
\node at (2,2) [circle,draw=black, fill=white, inner sep=0.5mm] {};
\end{tikzpicture}
\begin{tikzpicture}[scale=1]
\draw [ultra thick] (0,0) -- (0,1) ;
\draw [thick,dotted] (0,1) -- (0,2) ;
\draw [thick,dotted] (1,0) -- (1,1) ;
\draw [thick,dotted] (1,1) -- (1,2) ;
\draw [thick,dotted] (0,1) -- (1,1) ;
\draw [ultra thick] (1,1) -- (2,1) ;
\draw [thick,dotted] (0,2) -- (1,2) ;
\draw [ultra thick] (1,2) -- (2,2) ;
\draw [thick, dashed] (2,0) -- (2,1) -- (2,2);
\draw [thick, dashed] (0,0) -- (1,0) -- (2,0);
\node at (0,0) [circle,draw=black,fill=white,inner sep=0.5mm] {};
\node at (0,1) [circle,draw=black, fill=black, inner sep=0.5mm] {};
\node at (0,2) [circle,draw=black, fill=black, inner sep=0.5mm] {};
\node at (1,0) [circle,draw=black, fill=white, inner sep=0.5mm] {};
\node at (1,1) [circle,draw=black, fill=black, inner sep=0.5mm] {};
\node at (1,2) [circle,draw=black, fill=black, inner sep=0.5mm] {};
\node at (2,0) [circle,draw=black, fill=white, inner sep=0.5mm] {};
\node at (2,1) [circle,draw=black, fill=white, inner sep=0.5mm] {};
\node at (2,2) [circle,draw=black, fill=white, inner sep=0.5mm] {};
\end{tikzpicture}
\caption{Three distinct 1-spanning trees of the torus, out of the 32 total.}
\end{subfigure}

\vspace{1em}
\begin{subfigure}{\textwidth}
  \centering
\begin{tikzpicture}[scale=1] 
\draw [ultra thick] (0,0) -- (0,1) ;
\draw [ultra thick] (0,1) -- (0,2) ;
\draw [thick,dotted] (1,0) -- (1,1) ;
\draw [thick,dotted] (1,1) -- (1,2) ;
\draw [thick,dotted] (0,1) -- (1,1) ;
\draw [ultra thick] (1,1) -- (2,1) ;
\draw [ultra thick] (0,2) -- (1,2) ;
\draw [ultra thick] (1,2) -- (2,2) ;
\draw [thick, dashed] (2,0) -- (2,1) -- (2,2);
\draw [thick, dashed] (0,0) -- (1,0) -- (2,0);
\node at (0,0) [circle,draw=black,fill=white,inner sep=0.5mm] {};
\node at (0,1) [circle,draw=black, fill=black, inner sep=0.5mm] {};
\node at (0,2) [circle,draw=black, fill=black, inner sep=0.5mm] {};
\node at (1,0) [circle,draw=black, fill=white, inner sep=0.5mm] {};
\node at (1,1) [circle,draw=black, fill=black, inner sep=0.5mm] {};
\node at (1,2) [circle,draw=black, fill=black, inner sep=0.5mm] {};
\node at (2,0) [circle,draw=black, fill=white, inner sep=0.5mm] {};
\node at (2,1) [circle,draw=black, fill=white, inner sep=0.5mm] {};
\node at (2,2) [circle,draw=black, fill=white, inner sep=0.5mm] {};
\end{tikzpicture}
\begin{tikzpicture}[scale=1] 
\draw [thick,dotted] (0,0) -- (0,1) ;
\draw [thick,dotted] (0,1) -- (0,2) ;
\draw [ultra thick] (1,0) -- (1,1) ;
\draw [ultra thick] (1,1) -- (1,2) ;
\draw [ultra thick] (0,1) -- (1,1) ;
\draw [thick,dotted] (1,1) -- (2,1) ;
\draw [ultra thick] (0,2) -- (1,2) ;
\draw [ultra thick] (1,2) -- (2,2) ;
\draw [thick, dashed] (2,0) -- (2,1) -- (2,2);
\draw [thick, dashed] (0,0) -- (1,0) -- (2,0);
\node at (0,0) [circle,draw=black,fill=white,inner sep=0.5mm] {};
\node at (0,1) [circle,draw=black, fill=black, inner sep=0.5mm] {};
\node at (0,2) [circle,draw=black, fill=black, inner sep=0.5mm] {};
\node at (1,0) [circle,draw=black, fill=white, inner sep=0.5mm] {};
\node at (1,1) [circle,draw=black, fill=black, inner sep=0.5mm] {};
\node at (1,2) [circle,draw=black, fill=black, inner sep=0.5mm] {};
\node at (2,0) [circle,draw=black, fill=white, inner sep=0.5mm] {};
\node at (2,1) [circle,draw=black, fill=white, inner sep=0.5mm] {};
\node at (2,2) [circle,draw=black, fill=white, inner sep=0.5mm] {};
\end{tikzpicture}
\begin{tikzpicture}[scale=1] 
\draw [thick,dotted] (0,0) -- (0,1) ;
\draw [ultra thick] (0,1) -- (0,2) ;
\draw [ultra thick] (1,0) -- (1,1) ;
\draw [ultra thick] (1,1) -- (1,2) ;
\draw [ultra thick] (0,1) -- (1,1) ;
\draw [ultra thick] (1,1) -- (2,1) ;
\draw [thick,dotted] (0,2) -- (1,2) ;
\draw [thick,dotted] (1,2) -- (2,2) ;
\draw [thick, dashed] (2,0) -- (2,1) -- (2,2);
\draw [thick, dashed] (0,0) -- (1,0) -- (2,0);
\node at (0,0) [circle,draw=black,fill=white,inner sep=0.5mm] {};
\node at (0,1) [circle,draw=black, fill=black, inner sep=0.5mm] {};
\node at (0,2) [circle,draw=black, fill=black, inner sep=0.5mm] {};
\node at (1,0) [circle,draw=black, fill=white, inner sep=0.5mm] {};
\node at (1,1) [circle,draw=black, fill=black, inner sep=0.5mm] {};
\node at (1,2) [circle,draw=black, fill=black, inner sep=0.5mm] {};
\node at (2,0) [circle,draw=black, fill=white, inner sep=0.5mm] {};
\node at (2,1) [circle,draw=black, fill=white, inner sep=0.5mm] {};
\node at (2,2) [circle,draw=black, fill=white, inner sep=0.5mm] {};
\end{tikzpicture}
\caption{Three distinct 1-spanning co-trees of the torus, out of the 32 total.}
\end{subfigure}
  \caption{}
  \label{fig:trees_cotrees}
\end{figure}

\section{The process \label{sec:process}}
In this section, we construct a Markov CW chain given a system of weights on $X$. 
We continue to assume $X$ is a finite connected CW complex of dimension $d \ge 1$.
Our recipe makes use of Hypothesis \ref{hypo:decomp} in the case $k=d$.

\subsection{The cycle-incidence graph}

\begin{defn} For an integer $(d-1)$-cycle $z_0\in Z_{d-1}(X;\Z)$, let
  \[
    Z_{d-1}^{z_0}(X;\Z) = z_0 + B_{d-1}(X;\Z)
  \]
  denote the coset consisting of the
   integral $(d-1)$-cycles that are homologous to $z_0$.
\end{defn}

\begin{defn}
  \label{defn:cycle_incidence} Consider the 
directed graph $G$ defined as follows. The vertices of $G$ are given by 
integer $(d-1)$-cycles $z$ homologous to $z_0$, i.e.,
\[
  z \in Z_{d-1}^{z_0}(X;\Z)\, .
\]
A directed edge of $G$ with source $z$ is specified by a 4-tuple
\[
 e:=  (\alpha,f,\varepsilon_{\alpha,f},z)
\]
with 
$\alpha\in X_d$, $f \in X_{d-1}$, and $\varepsilon_{\alpha,f} \in X(\alpha,f)$, 
satisfying the following:
\begin{itemize}
\item $\langle z,f \rangle \ne 0$,
\item $\langle \partial \alpha,f\rangle = \sum_{\varepsilon \in X(\alpha,f)} 
  (-1)^{\chi(\varepsilon)} 
  \neq 0$, and
\item $z' = z - (-1)^{\chi(\varepsilon_{\alpha,f})}
\langle z,f\rangle \partial \alpha$.
\end{itemize}
In the above, the target of the edge $e$ is defined to be $z'$. 
To indicate this, we sometimes write
\[
z = s(e) \, , \qquad z' = t(e)\, .
\]

The {\it cycle-incidence graph} 
\[
\Gamma := \Gamma_{X,z_0}
\] 
is the directed subgraph of $G$ given by the directed path component of $z_0$. 
That is, a {\it vertex} $z$ lies in $\Gamma$ if
there exists a finite sequence of directed edges $z_0 \to z_1 \to \cdots
\to z_k \to z$, i.e., there is a finite directed path from $x$ 
to $z$.  An {\it edge} belongs to $\Gamma$ if and only if it
occurs in such a path.
\end{defn}

The cycle-incidence graph is the state diagram of 
the Markov CW chain described in the introduction, in which the cycle $z_0$ represents
 an initial condition. For a
particular choice of $(d-1)$-cell $f$ incident to $z$ and $d$-cell $\alpha$
incident to $f$, the cycle $z$ can `hop' across the $d$-cell $\alpha$, to form a new
cycle $z':= z - (-1)^{\chi(\varepsilon_{\alpha,f})}
\langle z,f\rangle \partial \alpha$.
This type of jump is known as an elementary transition. Informally,
an elementary transition consists of the cycle $z$ completely `jumping off' of the cell $f$ across $\alpha$ to form the new cycle $z'$ (cf.~Figures \ref{fig:intro_example} and \ref{fig:transition}). 

\begin{rem}
Typically, the newly formed cycle $z'$ will still have non-zero incidence
with the $(d-1)$-cell $f$ that is used in defining the elementary transition. 
There is one notable exception to this: 
when $|X(\alpha,f)| = 1$ we have $b_{\alpha,f} = \pm 1$. Consequently,
\begin{align*}
  \langle z',f \rangle &= \langle z, f \rangle - (-1)^{\chi(\varepsilon_{\alpha,f})}
  \langle z, f \rangle 
  \sum_k b_{\alpha,k} \langle k, f \rangle \\
  &= \langle z, f \rangle - (-1)^{2 \cdot \chi(\varepsilon_{\alpha,f})} \langle z, f \rangle = 0 \, .
\end{align*}
Therefore, in this case $z'$ will have trivial incidence with  $f$. 
\end{rem}

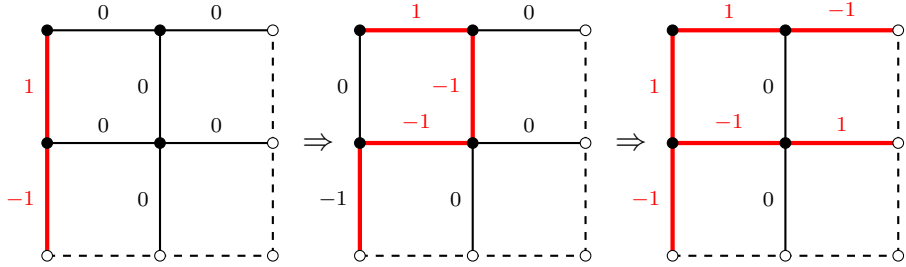
\begin{figure}
\centering
\begin{tikzpicture}[scale=1.5] 
\draw [ultra thick, red] (0,0) -- (0,1) node [midway,left] {\tiny $-1$};
\draw [ultra thick, red] (0,1) -- (0,2) node [midway,left] {\tiny $1$};
\draw [thick] (1,0) -- (1,1) node [midway, left] {\tiny $0$};
\draw [thick] (1,1) -- (1,2) node [midway, left] {\tiny $0$};
\draw [thick] (0,1) -- (1,1) node [midway, above] {\tiny $0$};
\draw [thick] (1,1) -- (2,1) node [midway,above] {\tiny $0$};
\draw [thick] (0,2) -- (1,2) node [midway, above] {\tiny $0$};
\draw [thick] (1,2) -- (2,2) node [midway, above] {\tiny $0$};
\draw [thick, dashed] (2,0) -- (2,1) -- (2,2);
\draw [thick, dashed] (0,0) -- (1,0) -- (2,0);
\node at (0,0) [circle,draw=black,fill=white,inner sep=0.5mm] {};
\node at (0,1) [circle,draw=black, fill=black, inner sep=0.5mm] {};
\node at (0,2) [circle,draw=black, fill=black, inner sep=0.5mm] {};
\node at (1,0) [circle,draw=black, fill=white, inner sep=0.5mm] {};
\node at (1,1) [circle,draw=black, fill=black, inner sep=0.5mm] {};
\node at (1,2) [circle,draw=black, fill=black, inner sep=0.5mm] {};
\node at (2,0) [circle,draw=black, fill=white, inner sep=0.5mm] {};
\node at (2,1) [circle,draw=black, fill=white, inner sep=0.5mm] {};
\node at (2,2) [circle,draw=black, fill=white, inner sep=0.5mm] {};
\end{tikzpicture}
\hspace{0.1em}
\raisebox{3.6em}{$\Rightarrow$} 
\hspace{-1.5em}
\begin{tikzpicture}[scale=1.5] 
\draw [ultra thick, red] (0,0) -- (0,1) node [midway,left,black] {\tiny $-1$};
\draw [thick] (0,1) -- (0,2) node [midway,left,black] {\tiny $0$};
\draw [thick] (1,0) -- (1,1) node [midway, left] {\tiny $0$};
\draw [ultra thick, red] (1,1) -- (1,2) node [midway, left] {\tiny $-1$};
\draw [ultra thick, red] (0,1) -- (1,1) node [midway, above] {\tiny $-1$};
\draw [thick] (1,1) -- (2,1) node [midway,above] {\tiny $0$};
\draw [ultra thick, red] (0,2) -- (1,2) node [midway, above] {\tiny $1$};
\draw [thick] (1,2) -- (2,2) node [midway, above] {\tiny $0$};
\draw [thick, dashed] (2,0) -- (2,1) -- (2,2);
\draw [thick, dashed] (0,0) -- (1,0) -- (2,0);
\node at (0,0) [circle,draw=black,fill=white,inner sep=0.5mm] {};
\node at (0,1) [circle,draw=black, fill=black, inner sep=0.5mm] {};
\node at (0,2) [circle,draw=black, fill=black, inner sep=0.5mm] {};
\node at (1,0) [circle,draw=black, fill=white, inner sep=0.5mm] {};
\node at (1,1) [circle,draw=black, fill=black, inner sep=0.5mm] {};
\node at (1,2) [circle,draw=black, fill=black, inner sep=0.5mm] {};
\node at (2,0) [circle,draw=black, fill=white, inner sep=0.5mm] {};
\node at (2,1) [circle,draw=black, fill=white, inner sep=0.5mm] {};
\node at (2,2) [circle,draw=black, fill=white, inner sep=0.5mm] {};
\end{tikzpicture}
\hspace{0.1em}
\raisebox{3.6em}{$\Rightarrow$}
\hspace{-1.5em}
\begin{tikzpicture}[scale=1.5] 
\draw [ultra thick, red] (0,0) -- (0,1) node [midway,left] {\tiny $-1$};
\draw [ultra thick, red] (0,1) -- (0,2) node [midway,left] {\tiny $1$};
\draw [thick] (1,0) -- (1,1) node [midway, left] {\tiny $0$};
\draw [thick] (1,1) -- (1,2) node [midway, left] {\tiny $0$};
\draw [ultra thick, red] (0,1) -- (1,1) node [midway, above] {\tiny $-1$};
\draw [ultra thick, red] (1,1) -- (2,1) node [midway,above] {\tiny $1$};
\draw [ultra thick, red] (0,2) -- (1,2) node [midway, above] {\tiny $1$};
\draw [ultra thick, red] (1,2) -- (2,2) node [midway, above] {\tiny $-1$};
\draw [thick, dashed] (2,0) -- (2,1) -- (2,2);
\draw [thick, dashed] (0,0) -- (1,0) -- (2,0);
\node at (0,0) [circle,draw=black,fill=white,inner sep=0.5mm] {};
\node at (0,1) [circle,draw=black, fill=black, inner sep=0.5mm] {};
\node at (0,2) [circle,draw=black, fill=black, inner sep=0.5mm] {};
\node at (1,0) [circle,draw=black, fill=white, inner sep=0.5mm] {};
\node at (1,1) [circle,draw=black, fill=black, inner sep=0.5mm] {};
\node at (1,2) [circle,draw=black, fill=black, inner sep=0.5mm] {};
\node at (2,0) [circle,draw=black, fill=white, inner sep=0.5mm] {};
\node at (2,1) [circle,draw=black, fill=white, inner sep=0.5mm] {};
\node at (2,2) [circle,draw=black, fill=white, inner sep=0.5mm] {};
\end{tikzpicture}
\caption{Two elementary transitions on the torus of Example~\ref{ex:torus}.
According to the orientations of Figure~\ref{fig:torus}, the initial cycle first
jumps across the 2-cell $A$ and then the 2-cell $B$, resulting in the 
displayed cycles.}
\label{fig:transition}
\end{figure}

When $\dim X =d=1$, it is not hard to identify the directed graph $\Gamma$ provided that the initial state is a vertex.
Define the {\it double} $DX$ of $X$ to be the directed graph 
with the same set of vertices, where a directed edge is specified by a pair
\[
(i,\alpha) \in X_0\times X_1
\]
such that $i$ is an endpoint $\alpha$. We also assume that
$\alpha$ has distinct endpoints.
We take the initial state $z_0$ to be any vertex of $X$. We
also remind the reader that $X$ is assumed to be finite and connected.

\begin{lem} \label{lem:cycle_double}
 With respect to the above assumptions, $\Gamma= DX$.
\end{lem}

\begin{rem}
 In this case, we are implicitly
taking the finite sets $X(\alpha,f)$ to be singletons since
$b_{\alpha,f} = \pm 1$ in the case of graphs.
\end{rem}

\begin{proof}[Proof of Lemma \ref{lem:cycle_double}]
Let $\alpha$ be an edge of $X$ and write $\partial \alpha = j-i$
for the value of the boundary operator at $\alpha$, where
$i$ and $j$ are distinct vertices given by the endpoints of $\alpha$. Then the directed edge $(i,\alpha)$ determines an elementary transition from $i$ to $j$ given by the equation
\[
j = i  + b_{\alpha,i} \partial \alpha 
\]
where in this case $b_{\alpha,i}  =-1$. Similarly $(j,\alpha)$ provides
 an elementary transition 
from $j$ to $i$ given by 
\[
i = j  + b_{\alpha,j}  \partial \alpha 
\]
where $b_{\alpha,j}  =+1$. It is straightforward to check that every
elementary transition with source/target $i$ is given by the above.  
Since the initial
state is a vertex, the above also shows 
every other state arising from a sequence
of elementary transitions is also a vertex.  Furthermore, as $X$ is connected, 
every vertex can be reached by such a sequence. It follows that $\Gamma =DX$. 
\end{proof}

\subsection{The rates} 
Let $(\tau_D,\gamma)$ be a driving protocol. Then
$\gamma(t) := (E(t),W(t))$ where $E\: X_{d-1} \to \Bbb R$ and
$W\: X_d \to \Bbb R$ are one-parameter families of weights. 
Let $\beta > 0$ be a real number representing inverse temperature.

Let $e = (\alpha,f,\epsilon_{\alpha,f},z)$ be a directed edge of $\Gamma$.  The number
\begin{equation} \label{eqn:transition-rate}
k_{\alpha,f}(t) := e^{\beta(E_f(t)-W_\alpha(t))}
\end{equation}
will be taken as  the transition
 rate along $e$ at time $t$. In what follows, we sometimes denote the pair $(\alpha,f)$ by  $(\alpha_e,f_e)$.
Let the collection of such rates be denoted by $k_\bullet$. Then, the pair
\[
(\Gamma,k_\bullet)
\]
completes the description of the Markov CW chain.

\subsection{The master equation} 
The rates give rise to a time-dependent evolution operator $\cal H$ 
operating on the vector space $C_0(\Gamma;\Bbb R)$ of 0-chains,  
where for $z\in \Gamma_0$ we have
\begin{equation} \label{eqn:formulaH-chain} 
\cal H(z) := 
\sum_{\substack{e\in \Gamma_1\\ s(e) = z}}  
k_{\alpha_e,f_e}  \cdot (z - t(e))\, .
\end{equation}
Note that the sum is finite
since the vertices of $\Gamma$ have finite valence.

Consider the obvious embedding $C_0(\Gamma;\Bbb R)\subset C^0(\Gamma;\Bbb R)$
from 0-chains to 0-cochains i.e., functions $\Gamma_0 \to \Bbb R$ which we regard as
``distributions.''
Extend $\cal H$ to act on $C^0(\Gamma;\Bbb R)$  as follows: given
a distribution $p\: \Gamma_0 \to \Bbb R$ define
\begin{equation} \label{eqn:formulaH} 
\cal H(p)(z) = \sum_{w\in \Gamma_0}\cal H_{z,w}p(w)\, ,
\end{equation}
where $\cal H_{z,w}$ denotes the $(z,w)$-matrix entry of $\cal H$.
Again, the sum is finite since the number of non-trivial entries in every row and column  is finite.

The evolution of the process is described by the
{\it master equation}
\begin{equation} \label{eqn:real_master}
\dot p = -\tau_D \cal H p\, ,\quad  p(0) = p_0 \, ,
\end{equation}
where $p(t)$ is a one-parameter family of 0-cochains. In what
follows, we choose the initial distribution $p_0$ to be:
\[
p_0(z) = \begin{cases} 1 \qquad & z=z_0\, ,\\
0 &\text{otherwise,}
\end{cases}
\]
where $z_0 \in \Gamma_0$ is a fixed vertex.
There are  
technical issues with equation \eqref{eqn:real_master}, since 
   $C^0(\Gamma;\Bbb R)$ is usually infinite dimensional.
Fortunately,  the formal solution to \eqref{eqn:real_master} can be described using perturbation theory. 



\begin{ex} \label{ex:dim-one-process} Assume $\dim X = d = 1$. Then by Lemma \ref{lem:cycle_double},  $\Gamma = DX$.
Hence $C^0(\Gamma;\Bbb R) \cong C_0(X;\Bbb R)$ canonically. We choose the initial
state $z_0$ to be any vertex of $X$. In this case $\cal H$ is identified with the biased Laplacian 
$\partial \partial_{E,W}$ acting on $C_0(X;\Bbb R)$ and the process coincides with
the one of  \cite{Chernyak:algtop}.
\end{ex}

\subsection{The trajectory space}
\label{sec:reduction}

Let $\Gamma$ and $k_\bullet$ be as above.
A {\it trajectory of length $n$} 
consists of a directed path 
\[
z_0 @> e_1 >> z_1 @> e_2 >> \cdots @>e_{n-1} >> z_{n-1} @>e_{n}>> z_n 
\]
together with jump times $0 = t_0 \le 
t_1 \leq t_1 \leq t_2 \leq \cdots \leq t_n$, where
$e_k = (\alpha_k,f_k, \varepsilon_k,z_{k})$ is a directed edge of $\Gamma$ from
$z_{k}$ to $z_{k+1}$. We use the notation
\[
(z_\bullet,e_\bullet,t_\bullet)
\]
to refer to this trajectory.

Define the 
 {\it escape rate} at a vertex 
 $z\in \Gamma_0$ over the interval $[t,t']$ by the expression
  \begin{gather}
    \label{eqn:escape_rate}
  u_{z}(t, t') = \exp \left(- \displaystyle \sum_{e \in
  X(z)}\int_{t}^{t'} 
  \tau_D \, k_{\alpha_e,f_e}(s) ds \right)  \, ,
\end{gather}
where $X(z)\subset X_1$ is the set of directed edges having terminus $z$,
and $k_{\alpha_e,f_e}$ is the transition rate 
 across the directed edge  $e $ (cf.\ Eq~\eqref{eqn:transition-rate}).

\begin{defn} With respect to the above,
the {\it probability density} of the trajectory 
$(z_\bullet,e_\bullet,t_\bullet)$ is 
\[
  f[z_\bullet,e_\bullet,t_\bullet] := 
  \prod_{m=1}^{n} u_{z_m}(t_{m-1},t_{m}) 
  \prod_{m=1}^{n} \tau_D \, k_{\alpha_{e_m},f_{e_m}}(t_m) \,.
\]
\end{defn}
Finally, given that the process is at state $z_0$ at time $0$,
 the probability that 
the process is in state $z$ at time $t$ is 
\begin{equation}
  \label{eqn:prob_trajectory}
  P[z;t] :=  \\
  \sum_{n=0}^{\infty} 
  \int_0^t \!\! \int_0^{t_n}\!\! \cdots\!\! \int_0^{t_2} \!\! dt_{1} \cdots dt_n
  \sum_{\substack{(z_\bullet,e_\bullet,t_\bullet)\\z_n = z}}  f[z_\bullet,e_\bullet,t_\bullet]  
\end{equation}
where the summation on the far right runs over the set of trajectories of length $n$ that 
begin in $z_1$ and terminate in 
$z$.

\begin{prop} \label{prop:prob-measure} For every $t \ge 0$,
  the function $z\mapsto P[z;t]$ is a probability distribution.
  \end{prop}

\begin{proof} We first show that series \eqref{eqn:prob_trajectory} converges.
Note that the series consists of positive terms.
As there are finitely many rates, and the numbers $t_j$ are bounded, it follows that the $n^{\mathrm{th}}$
term of \eqref{eqn:prob_trajectory} is bounded by 
\[
x^n \int_0^t \!\! \int_0^{t_n}\!\! \cdots\!\! \int_0^{t_2} \!\! dt_{1} \cdots dt_n 
= \frac{(xt)^n}{n!}\, ,
\]
for a judicious choice of $x > 0$ (which depends on $\beta$ and $\tau_D$). The series $\sum \frac{(xt)^n}{n!}$ converges to $e^{xt}$.
Hence, by the comparison test, the series \eqref{eqn:prob_trajectory}  converges.

To conclude the proof, we need to explain why the series
\[
\sum_z P[z;t]
\]
converges to 1.  The expression \eqref{eqn:prob_trajectory} arises
from perturbation theory. The idea is to show that the formal solution $\varrho(t)$
to the master equation
is a probability distribution for each $t$. Set $A := -\tau_D \cal H$. With this notation the master equation becomes $\dot p = Ap$ with $p(0) = p_0$.

Rewrite $A$ as $A^0 +  A^1$, where $A^0$ is the diagonal matrix and 
$A^1$  equals 0 along the diagonal. Now set $A_\epsilon = A^0 + \epsilon A^1$.  
We then consider the equation 
\begin{equation}\label{eqn:eps-p1}
\dot p= A_\epsilon p,\quad p(0) = p_0\, ,
\end{equation}
where we assume the solution has the form
\begin{equation}\label{eqn:eps-p2}
p^0 +\epsilon p^1 + \epsilon^2 p^2 + \cdots \, ,
\end{equation}
where $p^0$ solves the equation  $\dot q = A_0 q$.  If we substitute the expression \eqref{eqn:eps-p2} into the equation \eqref{eqn:eps-p1}, expand both sides, equate the coefficients of $\epsilon^j$ for $j = 0,1,\dots$, and set $\epsilon = 1$, we tediously but
straightforwardly arrive
at the expression \eqref{eqn:prob_trajectory}.
Hence,  \eqref{eqn:prob_trajectory} is the formal solution to the master equation
\eqref{eqn:real_master}.

Let $u$ be the row vector whose value at every vertex of $\Gamma$ is $1$. Multiplying both sides of \eqref{eqn:real_master} by $u$ on the left, we obtain
\[
u\cdot \dot p = u\cdot (Ap) = (u A)\cdot p = 0 \, ,
\]
since the sum of the entries in any column of $A$ vanishes. 
As $u\cdot \dot p =0$, we infer that the formal solution $p(t)$ is such that
$u\cdot p(t) =c$ for some constant $c$. Since $u\cdot p(0) =1$, it follows that $c=1$. Consequently, the formal solution $\varrho(t) := P[z,t]$ 
is a probability distribution for all $t$.
\end{proof}

The proof of Proposition \ref{prop:prob-measure} also established the following result.

\begin{cor}\label{cor:prob-measure} The function $P[z,t]$ is the 
solution to the
master equation \eqref{eqn:real_master}.
\end{cor}

\subsection{Expectation}  The {\it expectation} of a $0$-cochain
$p\: \Gamma_0 \to \Bbb R$ is the formal sum
\[
  \mathbb{E}[p] := \sum_{z\in \Gamma_0} p(z) z\, .
\]
We will give criteria for deciding when such an expression exists as an element of $Z_{d-1}(X;\Bbb R)$.

Let $\Sigma = 2^{\Gamma_0}$ be the
 $\sigma$-algebra  of all subsets of $\Gamma_0$. For fixed $b\in X_{d-1}$,  
 the function
 \[
 \mu_b (A) := \sum_{z \in A} \langle z,b\rangle, \qquad A\in \Sigma
 \]
is a signed measure on $(\Gamma_0,\Sigma)$.

For a  function $p\: X_{d-1}\to \Bbb R$, i.e., a 0-cochain, we consider the series
\begin{align}\label{eqn:b-converge}
 \sum_{z\in \Gamma_0} p(z) \langle z,b\rangle  &:= \int\!\! p \, d\mu_b\, \\
 & := \int\!\! p \, d\mu^+_b - \int\!\! p \, d\mu^-_b\, , \nonumber
\end{align}
(i.e., the Lebesgue integral over a discrete measure space), where 
$
\mu_b = \mu_b^+ - \mu_b^-
$
is the Hahn-Jordan measure decomposition~\cite{Halmos:Measure} of $\mu_b$, in which  $\mu_b^\pm$
are the unsigned measures
\begin{align*}
\mu_b^+(A) \, &:= \,  \sup \{\mu(B)\, |\, B\subset A\} \, , \\
\mu_b^-(A) \, &:= \,  \sup \{-\mu(B)\, |\, B\subset A\}\, .
\end{align*} 

\begin{defn} A 0-cochain $p\in C^0(\Gamma;\Bbb R)$ is {\it good} if the integrals
\[
\int \! p\, d\mu^\pm_b \, ,
\]
are finite for all $b\in X_{d-1}$. In particular,
the series \eqref{eqn:b-converge} converges for good
0-cochains $p$ and for every $b\in X_{d-1}$.
\end{defn}

\begin{lem} \label{lem:expect-converge} If $p$ is good, 
  then its expectation $\Bbb E[p]$
defines an element of $Z_{d-1}(X;\Bbb R)$.
\end{lem}

\begin{proof}  By the identity $z = \sum_b \langle z,b \rangle\, b$, we infer
\[
 \Bbb E[p] = \sum_{z}  p(z)z = \sum_b \left(\sum_z   p(z)\langle z,b \rangle\right) b\, ,
\]
where the outer summation is finite.  By hypothesis, the inner summation converges.
It follows that $\Bbb E[p]$ defines an element of $C_{d-1}(X;\Bbb R)$. But clearly, this element is a cycle.
\end{proof}

\begin{ex} Assume $\dim X = 1$ and choose $z_0$ to be any vertex of $X$. 
Then $\Gamma = DX$ by Lemma \ref{lem:cycle_double}
and every $p \in C^0(\Gamma;\Bbb R) \cong C_0(X;\Bbb R)$ is good. With respect to this
identification, the expectation $\Bbb E\: C^0(\Gamma;\Bbb R) \to C_0(X;\Bbb R)$ is the identity homomorphism.
\end{ex}

We now return to $\varrho(t) := P[z,t]$, the formal solution of the master equation \eqref{eqn:real_master}. Set
\[
\rho(t) := \Bbb E[\varrho(t)] \, ,
\]
so that $\rho(t)$ is the expected value of $P[z,t]$ with respect to $z$.

\begin{rem} \label{rem:norms}
  With the identity $z = \sum_{b} \langle z,b \rangle\, b$, we put the following norm 
  on cycles:
  \[
    ||z|| = \sum_{b\in X_{d-1}} | \langle z , b \rangle | \, . 
  \]
\end{rem}

\begin{lem} \label{lem:rho-converge} 
The 1-parameter of family of 0-cochains $\varrho(t)$ is good.
In particular, the expected value $\rho(t) =E[\varrho(t)]$ defines a 1-parameter family of
elements of  $Z_{d-1}(X;\Bbb R)$.
\end{lem}

\begin{proof} Note the inequality $  |\langle z,b \rangle| \le ||z||$ holds for every $b\in X_{d-1}$.
Set $\rho_t := \rho(t)$, and recall that we have a fixed a vertex $z_0 \in \Gamma_0$ in
defining $\varrho_t$ via~\eqref{eqn:real_master}. Then for each $t$, it will be  enough to prove that
the series
\[
\sum_{z \in \Gamma_0} \varrho_t(z) ||z||
\]
converges. We filter $z \in \Gamma_0$ by the number of edges in a minimal
path from $z$ to $z_0$; call this number $u(z)$. The previous
display can then be rewritten as
\[
\sum_{n=0}^\infty\sum_{u(z) = n} \varrho_t(z)||z|| \, .
\]

The graph $\Gamma$ possesses the following global finiteness property: there
is a number $c >0$ such that the valence of any vertex of $\Gamma$ is at most $c$.
In particular, the number of directed paths of length $n$ which start at a given vertex is at most $c^n$.  Using this observation, the proof of Proposition \ref{prop:prob-measure}, and Taylor's remainder theorem, there
is a $w>0$ (which depends on $t,c,\beta$ and $\tau_D$) such that
\[
\sum_{u(z) = n} \varrho_t(z) ||z||
   \,\, \le \,\,\frac{e^{w}w^n}{n!}\, .
\]
Consequently,
\[
\sum_{n=0}^\infty\sum_{u(z) = n}  \varrho_t(z)||z||  \,\, \le \,\,
e^{w}\sum_{n=0}^\infty 
\frac{w^n}{n!} = e^{2w}\, . \qedhere
\]
 \end{proof} 

\subsection{The dynamical equation}

\begin{defn} \label{defn} For a periodic driving protocol $(\tau_D,\gamma)$ with 
$\gamma(t) = (E(t),W(t))$, the {\it dynamical operator} 
\[
H(t)\: C_{d-1}(X;\Bbb R) \to C_{d-1}(X;\Bbb R) 
\]
is defined by
\[
H = \partial e^{-\beta W}\partial^\ast e^{\beta E}\, .
\]
\end{defn}

\begin{rem} If $f \in X_{d-1}$ then 
\begin{align} \label{eqn:equation-for-H}
H(f) \quad &= \quad
\sum_{\alpha\in X_d} k_{\alpha,f} \langle f,\partial \alpha\rangle\partial \alpha\\
& = \sum_{\substack{\alpha\in X_d\\\epsilon_{\alpha,f} \in X(\alpha,f)}} 
(-1)^{\chi(\epsilon_{\alpha,f})}k_{\alpha,f} \partial \alpha \nonumber \, ,
\end{align}
where $k_{\alpha,f} := e^{\beta (E_f -W_\alpha)}$.
\end{rem}

\begin{defn}  The {\it dynamical equation} is 
\begin{equation} \label{eqn:dynamical_eqn}
    \dot p = -\tau_D  H p \, . 
\end{equation}
\end{defn}

In the above, it is implicitly assumed that the initial value 
$p(0) \in C_{d-1}(X;\Bbb R)$ of a solution is a cycle representing a fixed homology class.

\begin{lem} \label{lem:H&H} If $p\in C^0(\Gamma;\Bbb R)$ is good, 
then $\Bbb E[\cal H(p)]$ defines an element of $Z_{d-1}(X;\Bbb R)$. Furthermore, the following identity holds formally:
\[
\Bbb E[\cal H(p)] = H(\Bbb E[p])\, .
\]
\end{lem}

\begin{proof} If $p$ is good, then $\Bbb E[p]$ converges
to an element of $Z_{d-1}(\Gamma;\Bbb R)$ (cf.~Lemma \ref{lem:expect-converge}).
The linear transformation $H$ is continuous since it acts on a finite dimensional vector space. It follows that $H(\Bbb E[p])$ also converges.
By a straightforward calculation using \eqref{eqn:formulaH} and
\eqref{eqn:equation-for-H}, both $\Bbb E[\cal H(p)]$ and $H(\Bbb E[p])$ 
are  given by the expression
\[
\sum_{z\in \Gamma_0}\sum_{\substack{e\in \Gamma_1\\z = s(e)}} (-1)^{\chi(\epsilon_{\alpha_e,f_e})} k_{\alpha_e,f_e}p(z) \langle f_e,z\rangle
 \partial \alpha_e \, . 
\]
Hence, $\Bbb E[\cal H(p)] = H(\Bbb E[p])$ and $\Bbb E[\cal H(p)]$ is convergent. In particular, $\Bbb E[\cal H(p)]$ defines an element of $Z_{d-1}(X;\Bbb R)$.
\end{proof}

From this last result we readily deduce Theorem \ref{bigthm:expectation}:

\begin{cor}\label{thm:dynamical_eqn} The family of cycles 
$\rho(t) \in Z_{d-1}(X;\Bbb R)$ is the unique solution to the dynamical equation \eqref{eqn:dynamical_eqn} having initial value $z_0$.
\end{cor}

\begin{proof} Set $P = P[z,t]$.  By  Lemma~\ref{lem:rho-converge}, 
$\rho(t) = \Bbb E[P[z,t]]$ converges and thus differentiation commutes 
with expectation. Application of Lemma \ref{lem:H&H} and Corollary \ref{cor:prob-measure} yields
 \[
\dot \rho = \Bbb E[\dot P] = \Bbb E[\cal H P] = H\Bbb E[P] = H\rho \, .\qedhere
\]
\end{proof}

\section{The adiabatic theorem \label{sec:adiabatic}}
In this section, we state and prove the Adiabatic Theorem (Theorem~\ref{bigthm:adiabatic}) 
for the Markov CW chain on $X$.
The adiabatic theorem states that for slow enough driving, a {\it periodic} solution to the 
dynamical equation exists and is unique. Our proof is similar to that of~\cite{Chernyak:algtop},
but modified appropriately to the higher dimensional setting.

\subsection{Formal solution}
The dynamical equation is a first order linear system of differential equations, and so 
specifying an initial condition guarantees the existence of a unique
solution~\cite{Arnold:ODE}. 
%
We introduce the {\it time-ordered exponential} $U(t,t_0)$ for $0 \leq t_0 \leq
t \leq 1$, which uniquely solves the initial value problem
\[
  \frac{d}{dt} U(t,t_0) = -\tau_D {H}(t) U(t,t_0) \quad \quad U(t_0, t_0) = I \, .
\]
Explicitly,
\[
  U(t,t_0) = \lim_{N \to \infty} e^{-\varepsilon \tau_D H(t_N)} 
  e^{-\varepsilon \tau_D H(t_{N-1})} \cdots e^{-\varepsilon \tau_D 
    H(t_0)} ,
\]
where $\varepsilon = t/N$ and $t_j = j \varepsilon$.
The expression
\[
 \rho(t) = U(t,0) \rho(0) = \left(\lim_{N \to \infty} 
 e^{-\varepsilon \tau_D H(t_N)} e^{-\varepsilon \tau_D H(t_{N-1})} 
 \cdots e^{-\varepsilon \tau_D H(t_0)}\right) \rho(0)
\]
gives the formal solution to the dynamical equation \eqref{eqn:dynamical_eqn} for $\rho(0) = z_0$.
The time-ordered exponential is often denoted
\[
  \hat T \exp \left( -\tau_D \int_{t_0}^t H(\tau) d\tau  
  \right) := U(t,t_0) \, ,
\]
in analogy with the solution to a one-dimensional differential equation.

\begin{defn}
For an operator $A \: V \to V$ on a finite dimensional real inner product space $V$, let
\[
  |A| := \sup_{v \neq 0} \frac{|Av|}{|v|} = \sup_{|v|=1}|Av|
\]
be the standard operator norm. If $A$ is self-adjoint, then $|A| =\lambda$, where
$\lambda$ is the maximum of the absolute value of  the eigenvalues of $A$.
\end{defn}

In what follows, we think of $U(t,t_0)$ as acting on $B_{d-1}(X;\R)$, where the latter is equipped with the norm
arising from the restriction of the modified inner product $\langle- , - \rangle_{E(t)}$. 

\begin{lem}\label{lem:evol-bound}
  Let $(\tau_D,\gamma)$ be a  driving protocol. There exists
  a positive constant $\lambda$ so that for all $t> t_0 \in [0,1]$,
  \[
    |U(t,t_0)| <  e^{-\lambda \tau_D (t-t_0)}\, .
  \]
\end{lem}

\begin{proof} 
For $t \in [0,1]$,
  let $A(t)= - H(t)$ acting on 
  $B_{d-1}(X;\R)$. Then $A(t)$ is 
  negative definite and self-adjoint with respect to the restriction of the inner
  product $\langle - , - \rangle_{E(t)}$ to   $B_{d-1}(X;\R)$.
By compactness there is a $\lambda >0$ such that $-\lambda$ is
 greater than or equal to all eigenvalues
of $A(t)$ for every $t \in [0,1]$.
 Let $C$ be the constant operator
  given by $Cv = -\tau_D \lambda v$ and let $U_C(t,t_0)$ be the evolution operator for $C$. Then
  \[
 |U_C(t,t_0)|
 = e^{-\lambda \tau_D (t-t_0)}\, .
 \] 
But clearly, $|U(t,t_0)| \le |U_C(t,t_0)|$. 
 \end{proof}

\begin{proof}[Proof of Theorem~\ref{bigthm:adiabatic}] 
  Write $\rho^B(t) = \rho^B(\gamma(t),\beta)$ for the time-dependent 
  Boltzmann distribution. Then $\rho^B$ is $1$-periodic, since $\gamma$ is. Let $\rho(t)$
  denote a solution to the dynamical equation Eq.~\eqref{eqn:dynamical_eqn}
  with initial value $\rho^B(0)$.
  Then 
  \[
  \rho(t) = \rho^B(t) + \xi(t)\, ,
    \]
   where $\xi\:[0,1] \to B_{d-1}(X;\R)$
   is a path. Hence, $\rho(t)$ is 1-periodic
  precisely when $\xi(t)$ is 1-periodic.
  Observe that $\xi(t)$ depends on $\tau_D$ whereas $\rho^B$ does not. 
  However, the values of $\rho(0)$ and $\xi(0)$ are independent of $\tau_D$.
  
   Apply the dynamical 
  operator to this solution. Then  the dynamical equation becomes
  \begin{gather}\label{eqn:pert-master}
    \dot \xi \,\, = \,\, -\tau_D H \xi - \dot \rho^B .
  \end{gather}
  The solution to equation \eqref{eqn:pert-master}
  is then
  \begin{gather}\label{eqn:pert-soln}
    \xi(t) \,\, = \,\, U(t,0)\xi(0) - \int_0^tU(t,t')\dot \rho^B dt' \, .
  \end{gather}
  Evaluating at $t=1$, the requirement for $\rho$ to be 1-periodic is equivalent to demanding
  that the equation
  \begin{gather*}
    \left( I - U(1,0) \right) \xi(0) \,\, = \,\, -\int_0^1U(1,t') \dot \rho^B dt' 
  \end{gather*}
is satisfied.
As $\tau_D$ is made large, the non-zero eigenvalues of $-\tau_D H(t)$
 tend to $-\infty$. Hence, by compactness, it follows that
there is a $\tau_0 >0$ such that the operator $I-U(1,0)$ is invertible for $\tau_D \ge \tau_0$.
 Then
  \begin{gather}\label{eqn:pert-unique}
    \xi(0) = -\left(I - U(1,0)\right)^{-1} \int_0^1U(1,t')\dot \rho^B(t') dt' \, .
  \end{gather}
In particular, the periodic solution $\rho(t)$ exists and is unique for $\tau_D \ge \tau_0$.

  As for the adiabatic limit, it suffices to show that $|\xi(t)| \to 0$
  as $\tau_D \to \infty$.  From Eq.~\eqref{eqn:pert-soln}, 
    we have
 \begin{align*}
    |\xi(t)| \, \, & \leq\,\,  |U(t,0)| |\xi(0)| + \int_0^t | U(t,t')| |\dot \rho^B(t')| dt' \, ,\\
    & \le \,\, e^{-\lambda\tau_D t}|\xi(0)|  +\int_0^t  e^{-\lambda \tau_D(t-t')} |\dot \rho^B(t')| dt' \, , \qquad 
    \text{by Lemma } \ref{lem:evol-bound},\\
    & \le \,\, e^{-\lambda\tau_D t}|\xi(0)|  +c \int_0^t  e^{-\lambda \tau_D(t-t')} dt' \, ,  \\
    & = \,\, e^{-\lambda\tau_D t}|\xi(0)|  + \tfrac{c(1- e^{-\lambda \tau_D})}{\lambda \tau_D}\, ,
   \end{align*}
   where $c \ge |\dot \rho^B(t')|$ is any upper bound for all $t' \in [0,1]$.
   Consequently, $|\xi(t)|\to 0$ when $\tau_D \to \infty$.
\end{proof}

\section{The low temperature limit \label{sec:temp}} 
For fixed $(E,W,\beta)$, the Boltzmann distribution can be regarded
as a homomorphism of vector spaces
\begin{align*}
\rho^B\:H_{d-1}(X;\Bbb R) &\to Z_{d-1}(X;\Bbb R)\, ,\\
x&\mapsto \rho^B(x) \, .
\end{align*}
Recall that $\rho^B$ is dependent on the parameters $(E,W,\beta)$.

If $E$ is one-to-one, then the functional
\[
L \mapsto \sum_{i \in L_{d-1}} E_i 
\]
has a unique minimum for some spanning co-tree $L^\mu$. 
In this case, we say $L^\mu$ is the {\it minimal} spanning co-tree for $E$.

\begin{lem}\label{lem:low_temp_boltz}
  Suppose $E$ is one-to-one. Then
  the low temperature  limit of $\rho^B$ 
  is supported on the minimal spanning co-tree $L^\mu$, i.e.,
  \[
    \lim_{\beta \to \infty} \rho_\beta^B \,\, = \,\, \psi_{L^\mu}\, ,
  \]
  and the convergence is uniform.
\end{lem}

\begin{proof} This follows from \cite[cor.~B]{CCK:Boltzmann}, 
but we now include some details.
 Since the domain of $\rho^B$ is compact, uniform convergence follows from
  pointwise convergence. 
  We proceed by studying the components of $\rho^B$ individually.

  Let $L$ be a spanning co-tree. Multiply the
  numerator and denominator of the component 
  \[
  \rho^B_L  := \frac{b_L}{\nabla} \psi_L
  \]
 by the expression
  $\exp{(-\beta \sum_{a\in L^{\mu}_{d-1}} E_a)}$ to obtain
  \begin{gather}\label{eqn:rhoB_modified}
    \rho_L^B = \frac{ \displaystyle{a_L^2 \exp\left\{-\beta \left(\sum_{b \in L_{d-1}} E_b - \sum_{a \in L^{\mu}_{d-1}}
E_a\right)\right\} \psi_L}}
{\displaystyle{\sum_K a_K^2 
  \exp\left\{ \displaystyle{-\beta \left( \sum_{e \in K_{d-1}} E_e - \sum_{a \in L^{\mu}} E_a \right)}\right\}}}
   \,\, .
  \end{gather}
  where the left-most sum in the denominator is taken over all spanning co-trees.
  Since $L^{\mu}$ is minimal, the numerator tends to zero for all $L \neq L^{\mu}$. When 
  $L=L^{\mu}$, the difference of sums vanishes and the numerator tends to 
  $a_{L^{\mu}}^2 \psi_{L^{\mu}}$.
  The same argument is true for the sum in the denominator, in which case we have
  \begin{gather*}
    \lim_{\beta \to \infty} \rho^B = \frac{a_{L^{\mu}}^2 \psi_{L^{\mu}}}{a_{L^{\mu}}^2} = 
    \psi_{L^{\mu}} \, .\qedhere
  \end{gather*}
\end{proof}

Similarly, if $W$  is one-to-one, then just as for spanning co-trees, the functional
on the set of spanning trees given by
\[
  T \mapsto \sum_{\alpha \in T} W_{\alpha}
\]
has a unique minimum $T^{\mu}$, henceforth called the 
{\it minimal} spanning tree. Recall from  Remark \ref{rem:ortho-section-bdy} that the operator
\[
\cal A = \tfrac{1}{\Delta} \sum w_T \varsigma_T
\]
is the orthogonal section of the boundary operator $\partial\: C_d(X;\Bbb R) \to B_{d-1}(X;\Bbb R)$ in the modified
inner product $\langle{-},{-} \rangle_W$.
Then an argument analogous to Lemma~\ref{lem:low_temp_boltz}, which we omit,
yields the following result.

  \begin{lem} \label{lem:low_temp_K}
   Assume $W$ is one-to-one. Then the low temperature
    limit of the operator
    $\cal A$ is supported on the minimal spanning tree $T^\mu$, i.e.,
    \[
      \lim_{\beta \to \infty} \cal A \,\, = \,\, \varsigma_T^{\mu} \, ,
    \]
   and the convergence is uniform.
  \end{lem}


 We now turn to the time-dependent case. Assume $(\tau_D,\gamma)$ is a driving
 protocol where $\gamma(t) = (E(t),W(t))$.

  \begin{prop}\label{prop:low_temp_deriveboltz}
  Let $L$ be a spanning co-tree and let $E$ be one-to-one
  for all $t$. The $L$-component of the 
  time derivative of the Boltzmann distribution tends to 0 uniformly 
  in the low temperature limit.
\end{prop}
\begin{proof}
  A tedious but straightforward computation of the time derivative of 
  Eq.~\eqref{eqn:rhoB} gives 
  \begin{gather}\label{eqn:deriv_rhoB}
    \dot\rho^B_L = \frac{\beta \,
      \displaystyle{a_L^2  \exp{(-\beta \sum_{b \in L} E_b)}
      \sum_K \left[ a_K^2  \exp{(-\beta \sum_{a \in K} E_a)} \left(\sum_{a \in K} \dot E_a - \sum_{b \in L} \dot E_b \right)\right]}}
      {\displaystyle{\left[ \sum_K a_K^2 \exp{\left(-\beta \sum_{a \in K}E_a \right)}\right]^2}} \, \psi_L \, .
  \end{gather}

  For convergence in the low temperature limit, we only need to verify the
  statement point-wise since $[0,1]$ is compact, and it suffices check the
  statement for each component $\dot\rho_L^B$.  First, multiply the numerator
  and denominator of Eq.~\eqref{eqn:deriv_rhoB} by $\exp\left\{-2 \sum_{b \in
  L} E_b\right\}$ to get
  \begin{gather}\label{eqn:deriv_help}
  \dot\rho^B_L = \frac{\beta \,
      \displaystyle{a_L^2  \left[ \sum_K a_K^2 \exp\left\{(-\beta (\sum_{a \in K} E_a - \sum_{b \in L} E_b)\right\}
       \left(\sum_{a \in K} \dot E_a - \sum_{b \in L} \dot E_b \right)\right]}}
      {\displaystyle{\left[ \sum_K a_K^2 \exp\left\{-\beta 
      (\sum_{a \in K}E_a -\sum_{b \in L} E_b) \right\}\right]^2}} \, \, .
\end{gather}
  There are
  two cases to consider: either $L$ is the minimal spanning co-tree or it is not.

  If $L$ is the minimal spanning co-tree, so that $\sum_{b \in L} E_b <
  \sum_{a \in K} E_a$ for every other spanning co-tree $K$, then the denominator of 
  Eq.~\eqref{eqn:deriv_help} is given by
  \[
    \left[ a_L^2 + \sum_{K \neq L} a_K^2  
    \exp{(-\beta(\sum_{a \in K} E_a - \sum_{b \in L} E_b)}
   )\right]^2 ,
  \]
  which tends to $a_L^4 < \infty$ as $\beta \to \infty$. As for the numerator
  of Eq.~\eqref{eqn:deriv_help}, when $L = K$, we have 
  $\sum_{a \in K} \dot E_a = \sum_{b \in L} \dot E_b$ 
  and the numerator is exactly zero. If $L \neq K$,
  then the exponential factor is negative and tends to zero as $\beta \to \infty$.

  If $L$ is not the minimal spanning co-tree, then some other spanning co-tree will be minimal. 
  Therefore, at least one of the exponents $-\beta (\sum E_a - \sum E_b)$ 
  will be positive. Since the denominator is squared, 
  Eq.~\eqref{eqn:deriv_rhoB} is dominated by 
  $A \beta / e^{B \beta}$ for some constants $A$ and $B$ with $B>0$ for large $\beta$. 
  It is easy to see this expression tends to zero as $\beta \to \infty$.
\end{proof}

\section{Current generation \label{sec:quantize}}
As above we fix a cycle $z_0 \in Z_{d-1}(X;\Z)$.  For
a periodic driving protocol $(\tau_D,\gamma)$, assume
 $\tau_D$ large enough
so a unique 1-periodic solution $\rho(t)$ to Eq.~\eqref{eqn:dynamical_eqn} exists (cf.\ Theorem~\ref{bigthm:adiabatic}).
Recall the biased coboundary operator
 $\partial^*_{E,W} = e^{-\beta W} \partial^* e^{\beta E} $.


\begin{defn}
  For a periodic driving protocol $(\tau_D, \gamma)$ and $\beta > 0$, the
  {\it current density} at $t\in [0,1]$ is defined as 
  \begin{equation}
    \label{eqn:emp_cur_dens}
    {\mathbf J}(t) :=  \tau_D\partial^*_{E,W}\rho(t)      \in C_d(X;\Bbb R)\, ,
\end{equation}
  where $\rho(t)$ is the unique periodic solution to the dynamical equation \eqref{eqn:dynamical_eqn}.
  The {average current} is
  \begin{gather} \label{eqn:average_cur}
    Q  = \int_0^{1} \! {\mathbf J}(t)dt \, .
  \end{gather}
\end{defn}

Note that $\bf J$ satisfies the continuity equation $\partial \bf J = -\dot \rho$.
When $\tau_D$ is sufficiently large,
$Q$ defines a real $d$-dimensional homology class.
To see this, apply $\partial$ to Eq.~\eqref{eqn:average_cur} to find
\begin{align*}
  \partial Q &= \tau_D \int_0^1 \partial \partial_{E,W}^* \rho(t) dt \\
  &= -\tau_D \int_0^1 \dot \rho \, dt \, \\
  &= \tau_D (\rho(0) - \rho(1))\\
  &= 0 \,,
\end{align*}
since $\rho$ is 1-periodic. Consequently, for $\tau_D$ sufficiently large,
 $Q$ is a $d$-cycle.

\begin{lem}
  \label{lem:current-K}
  The current density ${\mathbf J}$ coincides with the expression
\[
  \cal A(\dot\rho) \, ,
 \]
  where $\cal A$ is the operator of Eq.~\eqref{eqn:K_op} and $\rho$ 
  is the periodic solution of the dynamical equation.
\end{lem}

\begin{proof}
  Consider the set of all $w(t) \in C_d(X;\R)$, with $t \in [0,1]$ satisfying
  \begin{itemize}
    \item $\partial w = -\dot \rho$, and 
    \item $\langle w(t), z \rangle_{W(t)} = 0$ for all $z \in Z_{d}(X;\R)$, and $t\in [0,1]$.
  \end{itemize}
  Then any $w \in C_d(X;\R)$ satisfying
  the above two conditions is  necessarily unique. 
From the definition of ${\mathbf J}$, the first condition is verified
  by Eq.~\eqref{eqn:dynamical_eqn}, and the second condition follows from the definition of the
  modified inner product. 
  
  It therefore suffices to show that the above two conditions are satisfied
by the expression   $\cal A(\dot\rho)$. The first condition
  follows from the fact that $\cal A$ is a section of $\partial$, whereas the second is follows from the fact that $\cal A$ gives an orthogonal splitting.
\end{proof}

\begin{cor}
  If $\gamma$ is  constant, then $Q = 0$.
\end{cor}

\begin{proof}
  The weights appearing in $\cal A$ are time-independent since $\gamma$ is constant. 
  By Lemma~\ref{lem:current-K},
  \[
    Q = \int_0^1 \cal A( \dot \rho)\, dt 
    = \cal A( \int_0^1 \dot \rho\, dt) = \cal A(0) = 0 \, ,
  \]
  since $\rho$ is  1-periodic.
\end{proof}

\subsection{Quantization}
Current quantization occurs when the parameters are restricted to the generic
subspace of {\it good} parameters (compare \cite{Chernyak:algtop}). This space admits a decomposition
\[
\breve {\cal M}_X = U \cup V \, ,
\]
where $U$ denotes the subspace of parameters where $E$ is one-to-one,
and $V$ denotes the subspace where $W$ is one-to-one. Both $U$ and $V$ are
open subspaces. 

\begin{defn} Let $L \breve {\cal M}_X$ denote space of smooth unbased loops 
$\gamma\: [0,1]\to \breve {\cal M}_X$ in the Whitney $C^\infty$ topology. 
Such a $\gamma$ is called a {\it loop of good parameters}
and the pair $(\tau_D,\gamma)$ is called a {\it good driving protocol}.
\end{defn}

For a closed subinterval $I \subset [0,1]$,
the contribution along $I$ to the average current is given by the expression
\[
Q{\big|}_I = \int_{t\in I} {\mathbf J}(t)\, dt \, .
\]
We now choose a subdivision of $[0,1]$ such that the image of each segment under 
$\gamma$ lies in either $U$ or in $V$. More precisely, we choose
\[
0 = t_0 \le t_1 \le \cdots \le t_n = 1
\] 
a subdivision and set $I_j := [t_j,t_{j+1}]$. By taking the subdivision sufficiently fine and amalgamating
contiguous segments if necessary, we may assume that
\begin{enumerate}
  \item[(i)] $\gamma(I_j) \subset U$, or
  \item[(ii)] $\gamma(I_j) \subset V$ and $\gamma(\partial I_j) \subset U$,
\end{enumerate}
for every $j$.  The segments satisfying (i) are said to be of {\it type $U$}
and those satisfying (ii) are of {\it type $V$}. Then trivially
\begin{equation} \label{eqn:sum-decomp}
  Q \,\, =\, \, \sum_{k=0}^{n-1} \int_{I_k} {\mathbf J}(t)\, dt \,.
\end{equation}
Theorem~\ref{bigthm:adiabatic} implies 
$\lim_{\tau_D \to \infty} {\mathbf J}_{\tau_D} = \cal A( \dot \rho^B)$. 
Consequently, 
\begin{equation} \label{eqn:boltzmann-current-explicit}
  Q^B := \lim_{\tau_D \to \infty} Q(\tau_D,\beta) \,\, =\,\, 
  \int_0^1 \cal A( \dot \rho^B) \, dt\, .
\end{equation}


\begin{lem}
  Suppose that $I$ is of type $U$. In the low temperature limit,
  the contribution to $Q^B$ along $I$ is trivial.
\end{lem}

\begin{proof}
  By Lemma~\ref{lem:current-K} and \eqref{eqn:boltzmann-current-explicit}
  the average current along $I$ in the adiabatic limit is given 
  by 
  \[
   \int_I \cal A(\dot\rho^B) \, dt \,.
  \]
  Since $E$ is one-to-one on segments of type $U$, 
  Proposition~\ref{prop:low_temp_deriveboltz} implies that $\dot \rho^B \to 0$ uniformly
  in the low temperature limit. Consequently, $\cal A(\dot \rho)$ 
  also tends to zero.
\end{proof}

\begin{lem}\label{lem:current_V} 
  Suppose that $I=[u,v]$ is of type $V$. 
  In the low temperature limit, 
  the contribution to $Q^B$ along $I$ lies in 
  \[
    C_d(X; \Z[\tfrac{1}{\delta_I}]) \, ,
  \]
  where
  \[
  \delta_I\, :=\,  \theta_{T^\mu} a_{L^\mu(u)} a_{L^\mu(v)}\, ,
  \]
  in which
  \begin{itemize}
  \item $T^\mu$ is the unique minimal spanning tree on $I$\, ,
  \item $L^\mu(t)$ is the unique minimal 
spanning co-tree at $\gamma(t)$ for $t = u,v$, and
\item the integers $\theta_T$ and $a_L$ are defined in \S\ref{sec:combinatorial_struct}.
\end{itemize} 
\end{lem}

\begin{proof} 
By Lemma~\ref{lem:current-K}, the average current along $I$ in the adiabatic limit is given 
  by the expression  
   \[
 \int_I \cal A(\dot\rho^B) \, dt\, .
  \]
  Since $I$ is of type $V$, Lemma~\ref{lem:low_temp_K} implies that $\cal A \to \varsigma_T^\mu$ uniformly on $I$ as
  $\beta \to \infty$. 
  Therefore, the contribution to the low temperature limit of the Boltzmann current along $I$ is given by
  \begin{align*}
    \lim_{\beta \to \infty} Q^B_{\beta}{\big|}_I &=
     \lim_{\beta \to \infty} \int_u^v \cal A(\dot \rho^B)\, dt \, ,\\
    &= \varsigma_T^\mu \left( \lim_{\beta \to \infty} \int_u^v \dot \rho^B
   \, dt \right) \, ,\\
    &= \varsigma_T^\mu (\psi_{L^\mu(v)} - \psi_{L^\mu(u)} )[z_0] \,,\quad \text{by Lemma } \ref{lem:low_temp_boltz}. 
  \end{align*}
Note that the difference $\psi_{L^\mu(v)} - \psi_{L^\mu(u)}$
 takes image in  $B_{d-1}(X; \Z[\tfrac{1}{\delta_I}])$
  since its projection to $H_{d-1}(X; \Z[\tfrac{1}{\delta_I}])$ is trivial. Hence, the 
 displayed composition makes sense.
  
  By Remark~\ref{rem:D-coeffs}, we have a well-defined homomorphism
  \[
  \varsigma_{T^\mu} : B_{d-1}(X;\Z[\tfrac{1}{\delta_I}]) \to 
  C_d(X;\Z[\tfrac{1}{\delta_I}]) \, .
\]
Similarly, the same remark shows that the 
  difference 
  \[
    \psi_{L^\mu(v)} - \psi_{L^\mu(u)} \:   
    H_{d-1}(X; \Z[\tfrac{1}{\delta_I}]) \to B_{d-1}(X; \Z[\tfrac{1}{\delta_I}])  \, .
  \]
  is well-defined. Consequently, $\varsigma_T^\mu (\psi_{L^\mu(v)} - \psi_{L^\mu(u)} )$
is defined as a homomorphism  $H_{d-1}(X; \Z[\tfrac{1}{\delta_I}]) \to C_d(X;\Z[\tfrac{1}{\delta_I}])$. Applying this homomorphism to $[z_0]$ gives the conclusion. 
  \end{proof}

The following is now a straightforward consequence of the previous two lemmas
together with Remark~\ref{rem:D-coeffs}.

\begin{thm}[Quantization]\label{thm:quantization} Let $X$ be finite connected 
CW complex
$X$ of dimension $d$. Let $(\tau_D,\gamma)$ be a good driving protocol.

Then the low temperature, adiabatic limit of the average current $Q$ of 
$(\tau_D,\gamma)$ is well-defined and 
lies in  the fractional lattice 
\[
  H_d(X; \Z[\tfrac{1}{\delta}]) \subset H_d(X; \Bbb R)\, ,
\]
in which 
\[
  \delta :=   \prod_{L,T} a_L\theta_T  \, ,
  \]
 where $L$ ranges over all spanning co-trees in dimension $d-1$ and $T$ ranges
  over all spanning trees in dimension $d$.
\end{thm}

\begin{rem} As in \cite{CCK:Kirchhoff} and \cite{CCK:Boltzmann},
the factors appearing in $\delta$ have a
geometric and combinatorial significance. \end{rem}

\begin{rem}
  The space of good parameters can be extended to a space of {\it robust
  parameters} and one still obtains a quantization of the average current, as was
  done for graphs in~\cite{Chernyak:algtop}. This will be explained in future
  work~\cite{CCK:Hypercurrents}.
\end{rem}

\section{The example \label{sec:example}}
 Let $X$ and $\gamma = (E_\bullet,W_\bullet)$ be as in Theorem \ref{bigthm:ex}, where $\gamma$ is $1$-periodic. 
Figure \ref{fig:weights} indicates the ordering of the weights over the unit interval:

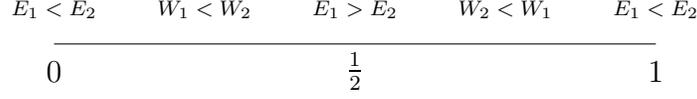
\begin{figure}[h]
  \centering
  \begin{tikzpicture}
    \draw (0,0) -- (8,0);
      \node [label={$0$}] at (0,-0.8) {};
    \node [label={$\tfrac{1}{2}$}] at (4,-0.9) {};
         \node [label={$1$}] at (8,-0.8) {};
        \node [label={\tiny $W_1 < W_2$}] at (2,0.05) {}; 
         \node [label={\tiny $W_2 < W_1$}] at (6,0.05) {}; 
         \node [label={\tiny $E_1 > E_2$}] at (4,0.05) {}; 
       \node [label={\tiny $E_1 < E_2$}] at (0,0.05) {}; 
              \node [label={\tiny $E_1 < E_2$}] at (8,0.05) {}; 
  \end{tikzpicture}
  \caption{The ordering of the weights over the unit interval}
  \label{fig:weights}
  \end{figure}
 Let $L_i\subset X$ be the $(d-1)$-sphere determined by 
 the $(d-1)$-cell $f_i$. Then $L_1,L_2$ are the spanning co-trees of $X$. Let $T_i \subset X$ be the spanning tree of $X$ given by attaching the $d$-cell $e_i$
 to the $(d-1)$-skeleton.
 
  The ordering $E_1 < E_2$ associates the spanning co-tree $L_1$ at $t=0,1$. The ordering $E_2 >E_1$ associates the spanning co-tree $L_2$ at $t=1/2$.
The ordering $W_1 < W_2$ associates the spanning tree $T_1$ on $(0,1/2)$ and the ordering $W_2 < W_1$ associates the spanning tree $T_2$ on $(1/2,1)$. Figure \ref{fig:co-tree-tree-over-time} gives the corresponding schematic with
spanning tree/co-tree labels replacing the inequalities of weights.

 \begin{figure}[h]
  \centering
  \begin{tikzpicture}
    \draw (0,0) -- (8,0);
      \node [label={$0$}] at (0,-0.8) {};
    \node [label={$\tfrac{1}{2}$}] at (4,-0.9) {};
         \node [label={$1$}] at (8,-0.8) {};
        \node [label={\tiny $T_1$}] at (2,0.05) {}; 
         \node [label={\tiny $T_2$}] at (6,0.05) {}; 
         \node [label={\tiny $L_2$}] at (4,0.05) {}; 
       \node [label={\tiny $L_1$}] at (0,0.05) {}; 
              \node [label={\tiny $L_1$}] at (8,0.05) {}; 
  \end{tikzpicture}
  \caption{The unit interval with spanning tree/co-tree labels}
  \label{fig:co-tree-tree-over-time}
  \end{figure}
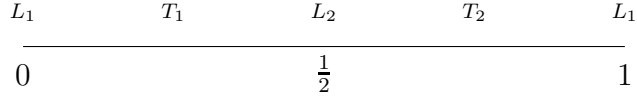

For the rest of the argument, we assume homology 
is taken with integer coefficients.
The generator of $H_{d-1}(X)$ is given by $f_i$ lying in $H_{d-1}(L_i) = Z_{d-1}(L_i)$, and by the definition of spanning co-tree these are uniquely defined.  
By the definition of spanning tree, there is a unique $d$-chain $e_1\in H_d(T_1) \subset C_d(T_1)$
 which bounds the difference $f_2-f_1  \in C_{d-1}(X)$ along $[0,1/2]$. Similarly,
 $-e_2\in H_d(T_2)$ uniquely bounds the difference $f_1-f_2$ 
 along $[1/2,1]$.  
 Then using equation \eqref{eqn:sum-decomp} and following the proof of Theorem \ref{bigthm:quant}, the average current in the low temperature, adiabatic limit is given by the sum of the two bounding $d$-chains, i.e.,
$e_1 + (-e_2) = e_1 - e_2$.  \qed

 \begin{rem}  In the case of the example,
 the space of parameters $\cal M_X$ is a real vector space of dimension
 four. The topological subspace of ``bad'' parameters has codimension two.  The low temperature adiabatic limit of the average current can be interpreted as the linking number of the good driving protocol $\gamma\: S^1 \to \cal M_X$ with the subspace of bad parameters.  
 \end{rem}

\end{document}